\documentclass{amsart}
\usepackage[utf8]{inputenc}
\usepackage{amsthm,amsmath,amssymb,amsfonts}
\usepackage[all]{xy}
\usepackage{color}
\usepackage[pagebackref=false, colorlinks=true, linkcolor=blue, citecolor=blue]{hyperref}
\usepackage{tikz-cd}

\numberwithin{equation}{section}

\newtheorem{theorem}{Theorem}[section]
\newtheorem{lemma}[theorem]{Lemma}

\newtheorem{corollary}[theorem]{Corollary}

\theoremstyle{definition}
\newtheorem{definition}[theorem]{Definition}
\newtheorem{example}[theorem]{Example}

\theoremstyle{remark}
\newtheorem{remark}[theorem]{Remark}

\newcommand{\sA}{\mathcal{A}}
\newcommand{\sC}{\mathcal{C}}
\newcommand{\sE}{\mathcal{E}}
\newcommand{\sF}{\mathcal{F}}
\newcommand{\sG}{\mathcal{G}}
\newcommand{\sL}{\mathcal{L}}
\newcommand{\sO}{\mathcal{O}}
\newcommand{\Oh}{\mathcal{O}}
\newcommand{\sU}{\mathcal{U}}

\newcommand{\sM}{\mathcal{M}}
\newcommand{\K}{\mathbb{K}}
\newcommand{\C}{\mathbb{C}}
\newcommand{\HH}{\mathbb{H}}

\newcommand{\de}{\partial}

\newcommand{\Der}{\operatorname{Der}}

\newcommand{\Tot}{\operatorname{Tot}}

\newcommand{\Ext}{\operatorname{Ext}}
\newcommand{\Hom}{\operatorname{Hom}}

\newcommand{\Spec}{\operatorname{Spec}}

\newcommand{\rank}{\operatorname{rank}}
\newcommand{\Id}{\operatorname{Id}}

\newcommand{\Tr}{\operatorname{Tr}}
\newcommand{\At}{\operatorname{At}}

\newcommand{\HOM}{\mathcal{H}om}
\newcommand{\END}{\mathcal{E}nd}

\newcommand{\contr}{{\mspace{1mu}\lrcorner\mspace{1.5mu}}}

\dedicatory{Dedicato ad Enrico, con amicizia e gratitudine}

\title{Semiregularity maps and deformations of modules over Lie algebroids} 

\author{Ruggero Bandiera}
\address{
Universit\`a degli studi di Roma La Sapienza,
Dipartimento di Matematica  Guido
Castelnuovo, P.le Aldo Moro 5,
I-00185 Roma, Italy.}
\email{bandiera@mat.uniroma1.it}

\author{Emma Lepri}
\address{School of Mathematics and Statistics,
University of Glasgow,
University Place,
Glasgow G12 8QQ, UK.}
\email{Emma.Lepri@glasgow.ac.uk}

\author{Marco Manetti}
\address{Universit\`a degli studi di Roma La Sapienza,
Dipartimento di Matematica  Guido
Castelnuovo,
P.le Aldo Moro 5,
I-00185 Roma, Italy.}
\email{marco.manetti@uniroma1.it}
\date{6 September, 2023}


\begin{document}

\begin{abstract}
We determine a DG-Lie algebra controlling deformations of a locally free module over a Lie algebroid $\sA$.
Moreover, for every  flat inclusion of Lie algebroids $\sA\subset \sL$ we introduce semiregularity maps and prove that they annihilate obstructions, provided that the Leray spectral sequence of the pair $(\sL,\sA)$ degenerates at $E_1$.
\end{abstract}

\maketitle

\section{Introduction}

Let $X$ be a  separated scheme of finite type over a field $\K$ of characteristic $0$ and 
let $\sE$ be a locally free sheaf on $X$. Following Buchweitz and Flenner \cite{BF}, the 
semiregularity maps of $\sE$ are defined as 
\[ \tau_k\colon \Ext^2_X(\sE,\sE)\to H^{2+k}(\Omega^k_X),\quad \tau_k(x)=\frac{1}{k!}\Tr(\At(\sE)^kx),\qquad k
\ge 0,\]
where $\At(\sE)\in \Ext^1_X(\sE,\sE\otimes\Omega^1_X)$ is the Atiyah class of $\sE$.

After \cite{ChS,BF,Pri} it is known that these semiregularity maps annihilate obstructions to deformations, provided that the Hodge to de Rham spectral sequence of $X$ degenerates at $E_1$. More generally, writing $\Omega^{\le k}_X$ for the algebraic de Rham complex truncated in degree $\le k$, it is known that the composition of $\tau_k$ with the natural map 
$H^{2+k}(\Omega^k_X)\to \HH^{2+2k}(\Omega^{\le k}_X)$ annihilates obstructions, regardless of degeneration properties of the aforementioned spectral sequence.

The main goal of this paper is to extend these results to locally free modules over a Lie algebroid 
$\sA$ on $X$, see  Definition~\ref{def.liealgebroid} below. By definition, a locally free $\sA$-module is a pair $(\sE,\nabla)$, where $\sE$ is a  locally free $\Oh_X$-module, and 
\[ \nabla\colon \sA\to \HOM_{\K}(\sE,\sE),\quad l\mapsto \nabla_l,\]
is an $\Oh_X$-linear map  such that: 
\begin{enumerate}
\item $\nabla$ is an $\sA$-connection; by definition, this means that 
$\nabla_l(fe) = a(l)(f)e+f\nabla_l(e)$ for $l\in\sA$, $f\in \Oh_X$ and $e\in \sE$, where 
$a\colon \sA\to \Theta_X$ is the anchor map;

\item the $\sA$-connection $\nabla$ is flat, i.e., its  
curvature $\nabla^2(l,m)=[\nabla_l,\nabla_m]-\nabla_{[l,m]}$ vanishes identically.
\end{enumerate}

When $\sA=\Theta_X$ with anchor map the identity, then the notion of $\sA$-connection reduces to   the usual definition of analytic connection.

Recall also that the Atiyah class of a locally free sheaf can be defined as the obstruction to the existence of an analytic connection.  In other words, the  Atiyah class of  $\sE$ can be defined as the obstruction to the lifting of the (unique) $0$-connection on $\sE$ to a $\Theta$-connection; in 
view of the  generalisation considered in this paper we also write $\At(\sE)=\At_{\Theta/0}(\sE)$.

By a straightforward generalisation, we can replace $\Theta$ with $\sA$ and define $\At_{\sA/0}(\sE)$ as the obstruction to the existence of an $\sA$-connection on $\sE$; however, this generalisation does not lead to anything new from the point of view of semiregularity maps and deformation theory. 

Instead, we are here interested in the definition of a class $\At_{\sL/\sA}(\sE)$ in the following situation:
\begin{enumerate} 
\item $\sA\subset \sL$ is an  inclusion of Lie algebroids such that the quotient sheaf
$\sL/\sA$ is locally free; 

\item $(\sE,\nabla)$ is a locally free $\sA$-module. 
\end{enumerate}

In the above situation the quotient  sheaf $\sL/\sA$ carries a natural structure of $\sA$-module given by the \emph{Bott connection} $\nabla^B\colon \sA\to \END_{\K}(\sL/\sA,\sL/\sA)$, 
$\nabla^B_a(x)=[a,x] \pmod{\sA}$. Thus, for every $r\ge 0$, 
the sheaf $\sM_r:=\bigwedge^r(\sL/\sA)^\vee\otimes\HOM_{\Oh_X}(\sE,\sE)$ carries a natural structure of 
$\sA$-module.

Denoting by 
$\HH^*(\sA; \sM_r)$ the Lie algebroid cohomology of $\sA$ with coefficients in $\sM_r$ (see Definition \ref{def.Lie_alg_coh}), in this paper we prove in particular that:
\begin{enumerate} 
\item $\HH^1(\sA; \sM_0)$ is the space of first order deformations of $\sE$ as an $\sA$-module;
\item $\HH^2(\sA; \sM_0)$ is a complete obstruction space for deformations of $\sE$ as an $\sA$-module;
\item the Atiyah class $\At_{\sL/\sA}(\sE)\in \HH^1(\sA; \sM_1)$ is properly defined.
\end{enumerate}
The first two items above are proved by showing that the DG-Lie algebra of derived sections of the sheaf of DG-Lie algebras  $\Omega^*(\sA)\otimes \HOM_{\Oh_X}(\sE,\sE)$ controls deformations of $\sE$ as 
an $\sA$-module, where  $\Omega^*(\sA)$ is the de Rham DG-algebra of $\sA$.
The Atiyah class $\At_{\sL/\sA}(\sE)$  is the primary obstruction to the extension of $\nabla$ to a flat $\sL$-connection. More precisely, $\At_{\sL/\sA}(\sE)$  is the obstruction  to the extension of $\nabla$ to an $\sL$-connection $\nabla'\colon \sL\to \HOM_{\K}(\sE,\sE)$
such that $[\nabla'_l,\nabla'_a]=\nabla'_{[l,a]}$ for every $l\in \sL$ and $a\in \sA$, cf. \cite{CSX}.

By analogy with the classical case, 
we define the semiregularity maps 
\[ \tau_k\colon \HH^2(\sA; \sM_0)\to \HH^{2+k}\bigg(\sA;{\bigwedge}^k(\sL/\sA)^\vee\bigg),\quad \tau_k(x)=\frac{1}{k!}
\Tr(\At_{\sL/\sA}(\sE)^kx),\]
and we use the main result of \cite{ChS} in order to prove that every $\tau_k$ annihilates obstructions, provided that the Leray spectral sequence (Definition~\ref{def.Leray_filtration})
of the pair $(\sL,\sA)$ degenerates at $E_1$.

\subsection{Notation} In this paper we work over a fixed field $\K$ of characteristic 0; unless otherwise specified every (graded) vector space is intended over $\K$. 

Unless otherwise specified the term \emph{differential graded} (DG) means graded over the integers and with differential of degree $+1$. The degree of a homogeneous element $x$ in a graded vector space will be denoted $|x|$.
We adopt the Grothendieck--Verdier formalism for degree shifting: given a DG-vector space 
$(V=\oplus_nV^n, d_V )$ and an integer $p$, we define the DG-vector space $(V[p], d_{V[p]})$ by setting
$V[p]^n=V^{n+p}$, $d_{V[p]}=(-1)^pd_V$. 

\bigskip
\section{Semiregularity maps for curved DG-algebras}

We briefly review some definitions and results from \cite{ChS}.
By a graded algebra we intend a unitary graded associative algebra over a fixed field $\K$ of characteristic $0$.  Every graded associative algebra is also a graded Lie algebra, with the bracket given by 
the graded commutator  $[a,b]=ab-(-1)^{|a||b|}ba$.

\begin{definition}\label{def.curved-algebra} A curved DG-algebra is the datum $(A,d,\cdot,R)$ of a graded associative algebra $(A,\cdot)$ together with a degree one derivation $d\colon A^*\to A^{*+1}$ and a degree two element $R\in A^2$, called \emph{curvature}, such that
	\[ d(R)=0,\qquad d^2(x)=[R,x]=R\cdot x-x\cdot R\quad\,\,\forall\, x\in A. \]
\end{definition}

For notational simplicity we shall write $(A,d,R)$ in place of $(A,d,\cdot,R)$ when the product $\cdot$ is clear from the context. We denote by $[A,A]\subset A$ the \emph{linear span} of all the graded commutators $[a,b]=ab-(-1)^{|a||b|}ba$.  Notice that $[A,A]$ is a homogeneous Lie ideal
and then $A/[A,A]$  inherits  a natural structure of DG-Lie algebra with trivial bracket.

\begin{definition}\label{def.curved-ideal} Let $A=(A,d,R)$ be a curved DG-algebra. 
A \emph{curved ideal} in $A$ is  homogeneous bilateral ideal 
$I\subset A$  such that $d(I)\subset I$ and $R\in I$. 

By a \emph{curved DG-pair} we mean the data $(A,I)$ of a curved DG-algebra $A$ equipped with a
curved ideal $I$. 
\end{definition}

In particular, for every curved DG-pair $(A,I)$, the quotient $A/I$ is a (non-curved) associative DG-algebra, and therefore also a DG-Lie algebra.
Writing $I^{(k)}$, $k\ge 0$, for the $k$th power of $I$, we have that 
$I^{(k)}$ is an associative bilateral ideal of $A$ for every $k$.
The differential graded algebra $\operatorname{Gr}_I A = \bigoplus_{k\ge 0} \dfrac{I^{(k)}}{I^{(k+1)}}$ is non-curved, since $d(I)\subset I$ and $d^2(I)\subset I^{(2)}$, the derivation $d$ factors through differentials
\[ d\colon \dfrac{I^{(k)}}{I^{(k+1)}} \to \dfrac{I^{(k)}}{I^{(k+1)}},\qquad d^2=0. \]

\begin{definition}\label{def.atiyahclass} Let $A=(A,d,R)$ be a curved DG-algebra and 
$I\subset A$ a curved ideal. 
The \emph{Atiyah cocycle} of the pair $(A,I)$ is  the class of $R$ in the DG-vector space 
$\dfrac{I}{I^{(2)}}$. The \emph{Atiyah class} of the pair $(A,I)$ is the cohomology 
class of the  Atiyah cocycle:
\[ \At(A,I)=[R]\in H^2\left(\dfrac{I}{I^{(2)}}\right).\]
\end{definition}

For every $x \in I$ of degree 1, we can consider the twisted derivation $d_x := d + [x,-]$ with curvature $R_x= R + dx + \dfrac{1}{2}[x,x]$. Then $I$ remains a curved ideal of the twisted 
 curved DG-algebra $(A,d_x, R_x)$.

\begin{lemma}\label{lem.curvatura_astratta}
	The Atiyah class of the pair $(A, d_x,R_x,I)$ does not depend on the choice of $x \in I$.
	The Atiyah class $\At(A,I)$ is trivial if and only if there exists $x \in I$ of degree 1 such that $R_x$ belongs to $I^{(2)}$.
\end{lemma}
\begin{proof}
	Firstly, notice that the differential on the algebra $\operatorname{Gr}_IA$ does not depend on the choice of $x \in I$: since $x$ belongs to $I$ the adjoint operator $[x,-]$ sends  $I^{(k)}$  to $I^{(k+1)}$, and so $d= d_x :=d + [x,-]$ in $\dfrac{I^{(k)}}{I^{(k+1)}}$. 
	In $\dfrac{I}{I^{(2)}}$, one has that $[x,x]=0$, so that
	\[ R_x- R = R + dx + \dfrac{1}{2}[x,x] - R= dx, \]
and the cohomology classes of $R$ and $R_x$ in $H^*\left(\dfrac{I}{I^{(2)}}\right)$ coincide.

Let now $x  \in I$ be such that $R_x = R + dx + \dfrac{1}{2}[x,x]$ belongs to $I^{(2)}$. Then $R + dx$ also belongs to  $I^{(2)}$ and $R = -dx $ in $\dfrac{I}{I^{(2)}}$, so that the Atiyah class is trivial. Conversely, let $R = dx$ in $\dfrac{I}{I^{(2)}}$, then $R- dx $ belongs to $I^{(2)}$, and so does $R_{-x} = R -dx +\dfrac{1}{2}[x,x]$.
\end{proof}

\begin{definition}\label{def.tracemap} A \emph{trace map} on a   
curved DG-algebra $(A,d,R)$ is the data of a complex of vector 
spaces $(C,\delta)$ and a morphism of graded vector spaces $\Tr\colon A\to C$ such that $\Tr\circ  d=\delta\circ \Tr$ and $\Tr([A,A])=0$.
\end{definition}

Assume now there are given a curved DG-algebra $(A,d,R)$, a curved ideal $I$ and a trace map $\Tr\colon A\to C$. Consider the decreasing filtration $C_k=\Tr(I^{(k)})$ of subcomplexes of $C$. By basic homological algebra, the spectral sequence associated to this filtration degenerates at $E_1$ if and only if for every $k$ the inclusion $C_k/C_{k+1}\subset C/C_{k+1}$ is injective in cohomology, see e.g. \cite[Thm. C.6.6]{LMDT}.  

In the above situation we can define semiregularity maps
\[ \tau_k\colon H^2(A/I)\to H^{2+2k}(C_k/C_{k+1}),\qquad \tau_k(x)=\frac{1}{k!}\Tr(\At(A,I)^kx).\]
The composition of $\tau_k$ with the natural morphism $H^{2+2k}(C_k/C_{k+1})\to H^{2+2k}(C/C_{k+1})$ is induced by the morphism of complexes
\[ \sigma_k^1\colon \frac{A}{I}\to \frac{C}{C_{k+1}}[2k],\quad \sigma_k^1(x)=\frac{1}{k!}\Tr(R^kx).\]
Considering $C/C_{k+1}$ as a DG-Lie algebra with trivial bracket, we can immediately see that $\sigma_k^1$ is a morphism of DG-Lie algebras for $k=0$, while for $k>0$ we have the following result.

\begin{theorem}[{\cite[Corollary 2.10]{ChS}}]\label{thm.ChS}
In the above situation, the map $\sigma_k^1$ is the linear component of an $L_{\infty}$-morphism 
$\sigma_k\colon A/I\rightsquigarrow C/{C_{k+1}}[2k]$. In particular, $\sigma_k^1$ annihilates obstructions for the deformation functor associated to the DG-Lie algebra $A/I$.
\end{theorem}

\bigskip
\section{Lie algebroid connections}
\label{sec.liealgebroid}

Throughout all this paper, $X$ will denote a smooth separated scheme of finite type over a field $\K$ of characteristic 0. 

We denote by $\Theta_X$ its tangent sheaf and by $\Omega^k_X$, $k\ge 0$, the sheaves of differential forms.   For every pair of sheaves of $\sO_X$-modules $\sF,\sG$ we denote by $\HOM_{\K}(\sF,\sG)$ and 
$\HOM_{\sO_X}(\sF,\sG)$ the sheaves of $\K$-linear morphisms and $\sO_X$-linear morphisms respectively. 
The $\sO_X$-module structure on $\sG$ induces an $\sO_X$-module structure
both on $\HOM_{\K}(\sF,\sG)$ and 
$\HOM_{\sO_X}(\sF,\sG)$.   We also write $\END_{\K}(\sF)$ and $\END_{\Oh_X}(\sF)$ for 
$\HOM_{\K}(\sF,\sF)$ and 
$\HOM_{\sO_X}(\sF,\sF)$ respectively.

Unless otherwise specified we write $\otimes$  for the tensor product over $\sO_X$, in particular for two $\sO_X$-modules $\sF,\sG$ we have $\sF\otimes\sG=\sF\otimes_{\sO_X}\sG$.

\begin{definition}\label{def.liealgebroid}
	A \emph{Lie algebroid} over $X$ is the data of $(\sL,[-,-],a)$ where:
	\begin{itemize} \item $\sL$ is a locally free coherent sheaf of $\sO_X$-modules;
		\item $[-,-]$ is a $\K$-linear Lie bracket on $\sL$;
		\item $a\colon\sL\to\Theta_X$ is a morphism of sheaves of $\sO_X$-modules, called the \emph{anchor map},  commuting with the brackets;
		\item finally, we require the Leibniz rule to hold
		\[   [ l , fm ] = a(l)(f)m + f[l,m],\qquad\forall\,l,m\in\sL, f\in\sO_X.  \]
	\end{itemize}
\end{definition}

\begin{example} The trivial sheaf $\sL=0$ and the tangent sheaf $\sL=\Theta_X$, with anchor map equal to the identity, are Lie algebroids. A Lie algebroid over $\operatorname{Spec} \K$ is exactly a Lie
	algebra over the field $\K$. Every sheaf of Lie algebras with $\Oh_X$-linear bracket can be considered as a Lie algebroid over $X$ with trivial anchor map.
\end{example}

\begin{example}[see \cite{DMcoppie} for details]\label{ex.coppie} Let $\sE$ be a locally free $\Oh_X$-module, then the sheaf of first order differential operators on $\sE$ with principal symbol has a natural structure of Lie algebroid. Since $\Theta_X$ is the sheaf of $\K$-linear derivations of $\Oh_X$, we can 
introduce the sheaf
\[ 
P(\Theta_X,\sE)=\{(\theta,\phi)\in \Theta_X\times \END_{\K}(\sE)\mid \phi(fe)=f\phi(e)+\theta(f)e,\; f\in \sO_X,\; e\in \sE\}.\]
Denoting by $a\colon P(\Theta_X,\sE)\to \Theta_X$ the projection on the first factor, we have an exact sequence of 
locally free $\sO_X$-modules
\[ 0\to  \END_{\sO_X}(\sE)\to P(\Theta_X,\sE)\xrightarrow{a}\Theta_X\to 0\]
and it is immediate to check that $P(\Theta_X,\sE)$ is a Lie algebroid with anchor map $a$.
Moreover, the map $P(\Theta_X,\sE)\to \END_{\K}(\sE)$, $(\theta,\phi)\mapsto \phi$, is injective
and its image is the sheaf of first order differential operators on $\sE$ with principal symbol.
\end{example}

The \emph{de Rham algebra  of $\sL$}  is defined as the sheaf of commutative graded algebras 
\[ \Omega^*(\sL)=\bigoplus_{k\ge 0}^{\rank\sL}\Omega^k(\sL),\qquad  \Omega^k(\sL)=\HOM^*_{\sO_X}(\sL[1]^{\odot k},\sO_X),\]
equipped with the convolution product. Notice that $\sL[1]$ is just $\sL$ considered as a graded sheaf 
concentrated in degree $-1$, hence $\Omega^*(\sL)$ is a locally free graded sheaf  with 
$\Omega^k(\sL)$ in degree $k$. By definition the convolution product is the dual of the coproduct 
$\Delta$  on the graded symmetric algebra $S(\sL[1])= \bigoplus_k \sL[1]^{\odot k}$, defined by
\[ \Delta(l_1,\ldots, l_n)=
\sum_{a=0}^{n} \sum_{\sigma\in S(a,n-a)}\!\!\epsilon(\sigma)
(l_{\sigma(1)},\ldots, l_{\sigma(a)})\otimes
(l_{\sigma(a+1)},\ldots, l_{\sigma(n)}),\]
where $\epsilon(\sigma)$ is the Koszul sign and $S(a,n-a)$ is the subset of unshuffles. 
More concretely, for  $\omega \in \Omega^k(\sL)$ and $\eta \in \Omega^j(\sL)$ we have 
\[ (\omega  \eta) (l_1, \ldots,l_{k+j}) 
\sum_{\sigma\in S(k,j)}\!\!(-1)^\sigma
\omega(l_{\sigma(1)},\ldots, l_{\sigma(k)})\eta(l_{\sigma(k+1)},\ldots ,l_{\sigma(k+j)}).\]
Notice that  the \emph{contraction product} 
\[ \sL\times\Omega^{k+1}(\sL)\xrightarrow{\contr} \Omega^k(\sL),\qquad (l\contr\omega)(l_1,\ldots,l_k)=
\omega(l,l_1,\ldots,l_k),\]
is $\sO_X$-bilinear and satisfies the Koszul  identity $l\contr(\omega\eta)=
(l\contr\omega)\eta+(-1)^{|\omega|}\omega(l\contr\eta)$.

More generally, if  $\sC^*$ is a sheaf of graded  associative $\sO_X$-algebras, the same holds for
\[ \Omega^*(\sL,\sC^*)=\Omega^*(\sL)\otimes\sC^*=\bigoplus_{k\ge 0} \HOM^*_{\sO_X}(\sL[1]^{\odot k},\sC^*).\]

The  \emph{de Rham differential of $\sL$}, denoted by $d_{\sL}\colon \Omega^k(\sL)\to \Omega^{k+1}(\sL)$, is defined by the formula (see e.g. \cite{MK}):
\[\begin{split} d_{\sL}(\omega)(l_0,\ldots,l_k)&=
\sum_{i=0}^n (-1)^{i} a(l_i)(\omega(l_0,\ldots,\widehat{l_i},\ldots,l_k))\\
&\quad+\sum_{i<j}(-1)^{i+j}\omega([l_i,l_j],l_0,\ldots,\widehat{l_i},\ldots,\widehat{l_j},\ldots,l_k).\end{split}\]

In particular for $\omega\in \Omega^0(\sL)=\sO_X$ we have 
$l\contr d_{\sL}(\omega)=d_{\sL}(\omega)(l)=a(l)(\omega)$, for every $l\in \sL$.
By definition $\Omega^k(\Theta_X)[k]=\Omega_X^k$ is 
the sheaf  of $k$-differential forms on $X$ and the global formula for the exterior derivative implies that  
$d_\Theta$ is the usual de Rham differential.

For every sheaf of $\sO_X$-modules $\sF$ we denote $\Omega^*(\sL,\sF)=\Omega^*(\sL)\otimes\sF$ and by 
\[\Omega^*(\sL)\times \Omega^*(\sL,\sF)\xrightarrow{\cdot} \Omega^*(\sL,\sF):\quad   
\eta\cdot (\sum_i \mu_i\otimes e_i)=\sum_i \eta\mu_i\otimes e_i,\quad \mu_i\in \Omega^*(\sL),\; e_i\in \sF,\]
\[\sL\times \Omega^*(\sL,\sF)\xrightarrow{\contr} \Omega^*(\sL,\sF):\quad   
l\contr (\sum_i \mu_i\otimes e_i)=\sum_i l\contr\mu_i\otimes e_i,\quad \mu_i\in \Omega^*(\sL),\; e_i\in \sF.\]

\begin{definition} Given  a  sheaf  of $\sO_X$-modules $\sF$, 
an $\sL$-connection $\nabla$ on $\sF$ is a $\K$-linear morphism of graded sheaves of degree $1$
\[ \nabla\colon \sF\to \Omega^1(\sL,\sF)=\Omega^1(\sL)\otimes \sF,\]
such that  
\[ \nabla(fe) = d_{\sL}(f)\cdot e+f\nabla(e),\qquad  \forall \ f\in\sO_X,\;  e\in\sF. \] 
\end{definition}

As in the usual case, every $\sL$-connection $\nabla$ admits a unique extension to $\K$-linear morphism of graded sheaves of $\sO_X$-modules of degree 
$1$ 
\[ \nabla\colon\Omega^*(\sL,\sF)\to \Omega^*(\sL,\sF)\]  
such that 
\begin{equation*}
	\nabla(f e) = d_{\sL}(f)\cdot e+(-1)^{|f|}f \nabla(e),\qquad \forall\; f\in\Omega^*(\sL),\;  e\in\Omega^*(\sL,\sF), 
\end{equation*}
and the connection is called flat if $\nabla^2=0$.

\begin{remark}\label{rem.derivata_covariante} Since the contraction product 
$\contr\colon \sL\times \Omega^1(\sL)\to \sO_X$ is  nondegenerate, every 
$\K$-linear morphism of sheaves $\nabla\colon \sF\to \Omega^1(\sL,\sF)$ is completely determined by 
the morphism of $\sO_X$-modules
 \[ \sL\to \HOM_{\K}(\sF,\sF),\quad l\mapsto \nabla_l:\qquad \nabla_l(e)=l\contr \nabla(e),\quad e\in \sF.\]
It is straightforward to verify that $\nabla$ is a connection if and only if 
\[ \nabla_l(fe) = a(l)(f)e+f\nabla_l(e),\qquad\forall\, f\in\sO_X,\; l\in\sL,\; e\in\sF. \] 
A simple computation shows that the curvature is given by the formula
\[ \nabla^2(l,m)(e)=\nabla_l\nabla_m(e)-\nabla_m\nabla_l(e)-\nabla_{[l,m]}(e),\qquad\forall\, l,m\in\sL,e\in\sF.\]

For instance, if $\sF$ is locally free and $\sL=\END_{\Oh_X}(\sF)$ (with trivial anchor map), then the natural inclusion  
$\sL\to \END_{\K}(\sF)$ is a flat connection.
\end{remark}

Since $\sL$ is locally free we have  natural isomorphisms 
\[ \Omega^*(\sL,\HOM_{\sO_X}(\sF,\sF))= 
\HOM_{\sO_X}^*(\sF,\Omega^*(\sL,\sF))=\HOM_{\Omega^*(\sL)}^*(\Omega^*(\sL,\sF),\Omega^*(\sL,\sF))\]
and, therefore, a natural identification of $\Omega^*(\sL,\HOM_{\sO_X}(\sF,\sF))$ with the subset of morphisms of graded sheaves 
$f\colon \Omega^*(\sL,\sF)\to \Omega^*(\sL,\sF)$ such that 
$f(\alpha\cdot\beta)=(-1)^{|f||\alpha|}\alpha\cdot f(\beta)$ for every $\alpha\in \Omega^*(\sL)$, 
$\beta\in \Omega^*(\sL,\sF)$.

The following lemma is a completely straightforward generalisation of  well known facts about connections and curvature.
 
\begin{lemma}\label{lem.sheafDGA}  Let $\nabla\colon\Omega^*(\sL,\sF)\to \Omega^*(\sL,\sF)$ be an $\sL$-connection, then 
$\nabla^2\in \Omega^2(\sL,\HOM_{\sO_X}(\sF,\sF))$ and 
$[\nabla,f]\in \Omega^*(\sL,\HOM_{\sO_X}(\sF,\sF))$ for every $f\in \Omega^*(\sL,\HOM_{\sO_X}(\sF,\sF))$.

In particular, 
$\left(\Omega^*(\sL,\HOM_{\sO_X}(\sF,\sF)), d=[\nabla,-], \nabla^2\right)$ is a properly defined sheaf of curved DG-algebras over $X$. 
\end{lemma}

If in addition $\sF$ admits a locally free resolution, then the trace map
$\Tr\colon \HOM_{\sO_X}(\sF,\sF)\to \Oh_X$, which is a morphism of sheaves of Lie algebras, is properly defined. By an analogous calculation to that of \cite[Lemma 2.6]{linfsemireg}, its extension 
\begin{equation}\label{eq.trace}\Tr\colon \Omega^*(\sL,\HOM_{\sO_X}(\sF,\sF))\to \Omega^*(\sL),\quad
\Tr(\omega\cdot f)=\omega\cdot \Tr(f),\; \omega\in \Omega^*(\sL),\; f\in 
\HOM_{\sO_X}(\sF,\sF),\end{equation}
is a trace map in the sense of Definition~\ref{def.tracemap}.

\begin{definition}
An $\sL$-module is a pair $(\sF,\nabla)$ consisting of a sheaf of $\sO_X$-modules $\sF$ and a flat $\sL$-connection $\nabla$ on $\sF$. An $\sL$-module $(\sF,\nabla)$ is said to be coherent (resp.: torsion free, locally free) if $\sF$ is coherent (resp.: torsion free, locally free)  as an $\Oh_X$-module.
\end{definition}

\begin{example} Every $\Oh_X$-module has a unique structure of module over the trivial Lie algebroid $\sL=0$.
\end{example}

\begin{example}\label{ex.oicse} For every Lie algebroid $\sL$, 
the pair $(\sO_X,d_{\sL})$ is an $\sL$-module. 
More generally every choice of a basis on a free
$\sO_X$-module gives an $\sL$-module structure.  
\end{example}

Every $\sL$-connection $\nabla$ on a locally free $\Oh_X$-module $\sF$ naturally induces  $\sL$-connections 
on the associated sheaves $\sF^\vee,\HOM_{\Oh_X}(\sF,\sF),\sF^{\wedge k}$ etc.. If
$\sF$ is an $\sL$-module, then also $\sF^\vee$, $\HOM_{\Oh_X}(\sF,\sF)$, $\sF^{\wedge k}$ etc. are $\sL$-modules in a natural way.

\begin{example} Let $(X,\pi)$ be a smooth Poisson variety, and denote by $\{-,-\}$ the Poisson bracket on the sheaf of functions $\sO_X$. The cotangent sheaf $\Omega^1_X$ of holomorphic differential 1-forms on $X$ has an induced structure of holomorphic Lie algebroid with the anchor $a(df):=\{f,-\}$ and the bracket $[df,dg]:=d\{f,g\}$ for all $f,g\in\sO_X$ (this defines $a$ and $[-,-]$ completely since $\Omega^1_X$ is generated by exact forms as an $\sO_X$-module), see e.g. \cite{FMpoisson} for more details. An $\Omega^1_X$-module is the same as a coherent sheaf $\sE$ together with a sheaf of Poisson modules structure on the sections of $\sE$. Namely, continuing to denote by $\{-,-\}$ the Poisson bracket on $\sE$, the associated connection is defined by
\[	\nabla\colon \Omega^1_X \to \END_{\K}(\sE),\qquad df \mapsto \nabla_{df}, \qquad\nabla_{df}e := \{f,e\} \quad \forall f \in \Oh_X, \ e \in \sE.\]
The fact that $\nabla$ is an $\Omega^1_X$-connection on $\sE$ is equivalent to the Poisson identities
\[ \{f,ge\} = \{f,g\}e+g\{f,e\},\qquad \{fg,e\}=f\{g,e\}+g\{f,e\},\]
while the flatness of $\nabla$ is equivalent to the Jacobi identity
\[ \{\{f,g\},e\}=\{f,\{g,e\}\}-\{g,\{f,e\}\}.\]
\end{example}

\begin{definition}\label{def.Lie_alg_coh}
	Let $\sL$ be a Lie algebroid over $X$. The hypercohomology of the complex $(\Omega^*(\sL), d_{\sL})$ is called the \emph{Lie algebroid cohomology of $\sL$}, and it is denoted by $\HH^*(\sL)$,
	
For  an $\sL$-module $(\sF, \nabla)$ the complex $(\Omega^*(\sL, \sF), \nabla)$ is called the \emph{standard complex} of $(\sF, \nabla)$ and its hypercohomology, denoted by $\HH^*(\sL; \sF)$,  is called the \emph{Lie algebroid cohomology of $\sL$ with coefficients in $\sF$}. 
\end{definition}

Notice that $\HH^*(\sL)=\HH^*(\sL;\Oh_X)$, where $\Oh_X$ carries the $\sL$-module structure of Example~\ref{ex.oicse}.
The notion of standard complex is borrowed from \cite{MK}, while for Lie algebroid cohomology we follow the notation of \cite{ABR, bru}.

\begin{example}
The Lie algebroid cohomology of the tangent sheaf $\Theta_X$ is the de Rham cohomology of $X$. The Lie algebroid cohomology of a Lie algebroid $\mathfrak{g}$ over  $\operatorname{Spec} \K$ is the Chevalley--Eilenberg cohomology of the Lie algebra $\mathfrak{g}$.
\end{example}

\bigskip
\section{Infinitesimal deformations of locally free $\sL$-modules}

In this section we describe a DG-Lie algebra controlling the infinitesimal deformations of a locally free $\sL$-module. In order to do so, we give a brief review of the Thom--Whitney totalisation.

Let $\sL$ be a Lie algebroid over $X$ and let $(\sE,\nabla)$ be an $\sL$-module, with $\sE$ locally free as an $\sO_X$-module.
Let $B$ be an Artin local  $\K$-algebra with residue field $\K$. We denote by $X_B=X\times\operatorname{Spec}(B)$,  by $p_X\colon X\times\operatorname{Spec}(B)\to X$ the projection onto the first factor, and by $\imath_X:X\to X\times\operatorname{Spec}(B)$ the inclusion induced by $B\rightarrow B/\mathfrak{m}_B=\K$. We notice that the pull-back sheaf $p^*_X\sL=\sL\otimes_{\K}B$ has a natural structure of Lie algebroid over $X_B$, with the Lie bracket extending $B$-bilinearly the one on $\sL$. Moreover, it is easy to check that a $p^*_X\sL$-module $\sF$ on $X_B$ restricts to an $\sL$-module $\imath^*_X\sF$ on the central fibre $X$. 
\begin{definition}
	A deformation of the $\sL$-module $(\sE,\nabla)$ over $\operatorname{Spec}(B)$ consists of the data of a deformation $\sE_B$ of $\sE$ over $X_B$ and a $p^*_X\sL$-module structure 
\[\nabla_B\colon \sE_B\to \Omega^1(p^*_X\sL,\sE_B)=\Omega^1(\sL)\otimes_{\sO_X}\sE_B\] 
such that the restriction $\imath^*_X\sE_B$ to $X$, with the naturally induced $\sL$-module structure, coincides with $(\sE,\nabla)$.
An isomorphism of deformations $(\sE_B,\nabla_B)\to (\sE_B',\nabla_B')$ is an isomorphism of deformations of sheaves 
$\phi\colon \sE_B\to \sE_B'$ such that $\phi\nabla_B=\nabla_B'\phi$. 
\end{definition}
We want to describe a DG-Lie algebra controlling the infinitesimal deformations of 
$(\sE,\nabla)$. To this end we first 
review  the definition and some of the main properties of the Thom--Whitney totalisation functor $\Tot$; for more details see e.g. \cite{sheaves,semicos,semireg, LMDT}.
   	The Thom--Whitney totalisation is a functor  from the category of semicosimplicial DG-vector spaces to the category of DG-vector spaces.
  For every $n \geq 0$ consider
   \[ A_n = \frac{\K[t_0,  \ldots, t_n,dt_0, \ldots, dt_n]}{(1- \sum_i t_i , \sum_i dt_i)}\]
   the commutative differential graded algebra of polynomial differential forms on the affine standard $n$-simplex, and the maps
   \begin{equation*}
   \delta_k^* \colon A_n \to A_{n-1},\ \ 0\leq k \leq n\quad \quad\ \ \delta_k^*(t_i)=
   \begin{cases}
   	t_i\quad i<k\\
   	0 \quad i=k\\
   	t_{i-1} \quad i>k.
   \end{cases}
  \end{equation*}
  \begin{definition}\label{def.tot}
The Thom--Whitney totalisation of a semicosimplicial DG-vector space  $V$
\begin{center}
	\begin{tikzcd}
	V: & V_0 \arrow[r, "\delta_0", shift left] \arrow[r, "\delta_1"', shift right] & V_1 \arrow[r, "\delta_0", shift left=2] \arrow[r, "\delta_2"', shift right=2] \arrow[r, "\delta_1" description] & V_2 \arrow[r, shift left=3] \arrow[r, shift left] \arrow[r, shift right=3] \arrow[r, shift right] & \cdots
	\end{tikzcd}
\end{center}
is the DG-vector space
\[ \Tot(V)= \big\{ (x_n) \in \prod_{n\geq 0} A_n \otimes_{\K} V_n \ | \ (\delta^*_k \otimes \Id )x_n = (\Id \otimes \delta_k) x_{n-1} \text{ for every } 0 \leq k \leq n \big\} ,\]
with differential induced by the one on $\prod_{n\geq 0} A_n \otimes V_n$. 

If $f\colon V \to W$ is a morphism of semicosimplicial DG-vector spaces, then $\Tot(f) \colon \Tot(V) \to \Tot(W)$ is defined as the restriction of the map
\[ \prod \Id \otimes f \colon \prod_{n\geq 0} A_n \otimes_{\K} V_n \to \prod_{n\geq 0} A_n \otimes_{\K} W_n .\]
\end{definition}

The $\Tot$ functor is exact: given semicosimplicial DG-vector spaces $V,W, Z$ and morphisms
$f\colon V\to W$, $g \colon W \to Z$ such that for every $n \geq 0$ the sequence
\begin{center}
	\begin{tikzcd}
	0 \arrow[r] & V_n \arrow[r, "f"] & W_n \arrow[r, "g"] & Z_n \arrow[r] & 0
	\end{tikzcd}
\end{center}
is exact, one obtains an exact sequence
\begin{center}
	\begin{tikzcd}
	0 \arrow[r] & \Tot(V) \arrow[r, "f"] & \Tot(W) \arrow[r, "g"] & \Tot(Z) \arrow[r] & 0,
	\end{tikzcd}
\end{center}
see e.g. \cite{FHT, LMDT}.

 Given two semicosimplicial DG-vector spaces $V$ and $W$, then $\Tot(V \times W)$ is naturally isomorphic to $\Tot(V)\times \Tot(W)$. An important consequence is the preservation of multiplicative structures; in particular, we will use the fact that the functor $\Tot$ sends semicosimplicial DG-Lie algebras 
to DG-Lie algebras.

\begin{example}\label{ex.cech}
	Let $(\mathcal{E}^*,\delta)$ be a bounded below complex of quasi-coherent sheaves on $X$, and let $\mathcal{U}= \{ U_i\}$ be  an open affine cover of $X$. Denote by $U_{i_1 \cdots i_n}= U_{i_1}\cap \cdots \cap U_{i_n}$, and  consider the semicosimplicial DG-vector space of \v{C}ech cochains:
	\begin{center}
		\begin{tikzcd}
		\mathcal{E}^*(\mathcal{U}): \quad \prod\limits_i \mathcal{E}^*(U_i) \arrow[r, shift left, "\delta_0"] \arrow[r, shift right, "\delta_1"'] & {\prod\limits_{i,j}\mathcal{E}^*(U_{ij})} \arrow[r, shift left=2, "\delta_0"] \arrow[r, shift right=2, "\delta_2"'] \arrow[r, "\delta_1" description] & {\prod\limits_{i,j,k}\mathcal{E}^*(U_{ijk})} \arrow[r, shift left=3] \arrow[r, shift left] \arrow[r, shift right=3] \arrow[r, shift right] & \cdots.
		\end{tikzcd}
	\end{center}
	
According to  Whitney integration theorem, there exists a natural quasi-isomorphism 
\[ I\colon \Tot(\mathcal{U},\mathcal{E}^*)\to C^*(\sU, \mathcal{E}^*)\]
where $C^*(\sU, \mathcal{E}^*)=\oplus_i C^*(\sU, \mathcal{E}^i)[-i]$ is the hypercomplex 
of \v{C}ech cochains (see \cite{whitney} for the $C^\infty$ version, \cite{getzler04,lunardon,LMDT,navarro} for the algebraic version used here). 
Therefore the cohomology of $\Tot(\mathcal{U},\mathcal{E}^*)$ is isomorphic to the hypercohomology of the complex of sheaves $\mathcal{E}^*$  
and then the quasi-isomorphism class  of $\Tot(\mathcal{U},\mathcal{E}^*)$ does not depend on the affine open cover, since $H^i(\Tot(\mathcal{U},\mathcal{E}^*))=\HH^i(X,\mathcal{E}^*)$ and the 
map $I$ commutes with refinements of affine covers.

For our later application it is important to point out that there exists a natural inclusion of 
 DG-vector spaces $\Gamma(X, \mathcal{E}^*) \to \Tot(\mathcal{U},\mathcal{E}^*)$ such that the restriction of $I$ to $\Gamma(X, \mathcal{E}^*)$ is the natural inclusion map
\[i\colon \Gamma(X, \mathcal{E}^*) \to \prod\limits_i \mathcal{E}^*(U_i),\qquad i(s)=\{s_{|U_i}\}. \]
In fact, $\delta_0 i= \delta_1 i $, therefore 
\[ \delta_{j_k}\delta_{j_{k-1}}\cdots\delta_{i_1}i= \delta_0^k i,\quad \text{for every } 0\le j_s\le s,\]and this implies that 
\begin{equation}\label{equ.inclusionsecriontot} 
\iota \colon \Gamma(X, \mathcal{E}^*) \to \Tot(\mathcal{U},\mathcal{E}^*),\quad
\iota(a)=(1 \otimes i(a), 1 \otimes \delta_0 i(a), 1 \otimes \delta_0^2 i(a), \ldots)\end{equation}
is a properly defined injective morphism of DG-vector spaces.

For later use we point out that for every quasi-coherent sheaf $\sF$ and every affine open cover 
$\sU$, the inclusion $\Gamma(X,\sF)\subset \Tot(\mathcal{U},\sF)$ induces an isomorphism 
$\Gamma(X,\sF)\cong H^0(\Tot(\mathcal{U},\sF))$.
\end{example}

Returning to our initial situation of a locally free $\sL$-module $(\sE,\nabla)$, since the $\sL$-connection $\nabla$ is flat,
 by Lemma~\ref{lem.sheafDGA} $\left(\Omega^*(\sL,\HOM_{\sO_X}(\sE,\sE)), d=[\nabla,-]\right)$ is a sheaf of locally free DG-algebras, which gives rise to a sheaf of locally free DG-Lie algebras
$(\Omega^*(\sL,\HOM_{\sO_X}(\sE,\sE)),d=[\nabla, -],[-,-])$. 

\begin{theorem}\label{thm.Deformazioni} In the above situation, for every affine open cover $\sU=\{U_i\}$, the DG-Lie algebra 
$\Tot(\sU, \Omega^*(\sL,\HOM_{\sO_X}(\sE,\sE)))$ controls the infinitesimal deformations of $(\sE,\nabla)$.
In particular $\HH^1(\sL;\HOM_{\sO_X}(\sE,\sE))$ is the space of first order deformations and 
$\HH^2(\sL;\HOM_{\sO_X}(\sE,\sE))$ is an obstruction space.
\end{theorem}

\begin{proof} This result is probably well known to experts, at least in the case $\sL=\Theta_X$, cf. \cite[Thm. 6.8]{GoMil1}, and follows easily from Hinich's theorem on descent of Deligne groupoids.
According to \cite{hinichdescent}, it is sufficient to check that locally 
the Deligne groupoid of $\Omega^*(\sL,\HOM_{\sO_X}(\sE,\sE))$ is equivalent to the groupoid of 
deformations of  $(\sE,\nabla)$. 

In order to check this, it is not restrictive to assume $X$ affine. 
Given an Artin ring $B$ as above, up to isomorphism every deformation of 
$\sE$ is trivial, i.e. $\sE_B=\sE\otimes_{\K}B$ and 
$\HOM_{\sO_{X_B}}(\sE_B,\sE_B)=\HOM_{\sO_X}(\sE,\sE)\otimes_{\K}B$.
Denoting by $\nabla_0\colon \sE_B\to \Omega^1(p^*_X\sL,\sE_B)=\Omega^1(\sL,\sE)\otimes_{\K}B$ the natural $B$-linear extension of $\nabla$, every deformation of $\nabla$ over $B$ is of the form
$\nabla_0+x$, with $x\in \Gamma(X,\Omega^1(\sL,\HOM_{\sO_X}(\sE,\sE)))\otimes_\K\mathfrak{m}_B$, and the flatness condition 
$(\nabla_0+x)^2=0$ is exactly the Maurer--Cartan equation $dx+\frac{1}{2}[x,x]=0$.

To conclude the proof we only need to show that two solutions of the Maurer--Cartan equation $x,y$ are gauge equivalent if and only if there exists an isomorphism of deformations $\phi\colon \sE_B\to \sE_B$ such that 
$\phi(\nabla_0+x)\phi^{-1}=\nabla_0+y$. Every $\phi$ as above is of the form $\phi=e^a$, with 
$a\in \Gamma(X,\HOM_{\sO_X}(\sE,\sE))\otimes_\K\mathfrak{m}_B$, and then the condition 
$\phi(\nabla_0+x)\phi^{-1}=\nabla_0+y$ is equivalent to 
\[ \nabla_0+y=e^{[a,-]}(\nabla_0+x)=\nabla_0+x+\sum_{n=0}^\infty \frac{[a,-]^n}{(n+1)!}([a,x]-da),\]
which is the same as $y=e^a\ast x$, where $*$ denotes the gauge action.
\end{proof}

\begin{remark}\label{rem.deformazioni} One can consider a different deformation problem, namely the deformation of pairs (bundle, $\sL$-connection) without requiring the vanishing of the curvature. Then the same argument as above shows that this deformation problem is controlled by the DG-Lie algebra  
$\Tot(\sU, \Omega^{\le 1}(\sL,\HOM_{\sO_X}(\sE,\sE)))$, while it is well known that 
$\Tot(\sU,\HOM_{\sO_X}(\sE,\sE))$ controls the deformations of $\sE$ \cite{sheaves}.
\end{remark}

\bigskip
\section{Lie pairs}\label{sec.lie_pairs}

\begin{definition} A \emph{Lie pair} $(\sL,\sA)$ of Lie algebroids over $X$ is a pair consisting of a Lie algebroid $\sL$ over $X$ and a Lie subalgebroid $\sA\subset\sL$ such that the quotient sheaf $\sL/\sA$ is  locally free.
\end{definition}

Let $(\sL,\sA)$  be a Lie pair. Since $\sL/\sA$ is assumed locally free we have a surjective restriction map 
$\varrho\colon\Omega^*(\sL)\to \Omega^*(\sA)$, which is a morphism of sheaves of commutative differential graded algebras.
The powers of its kernel give a finite decreasing filtration of differential graded ideal sheaves 
\[ \Omega^*(\sL)=\sG_0^*\supset \sG_1^*=\ker(\varrho)\supset\cdots \sG_r^*=(\ker(\varrho))^{(r)}\supset\cdots.\]
If we forget the de Rham differential, we can immediately see that $\sG_p^*$ is the image of the morphism of graded 
$\Oh_X$-modules 
\[ \bigwedge^p\left({\sL}/{\sA}\right)^\vee[-p]\otimes \Omega^*(\sL)\to \Omega^*(\sL),\]
and we have natural isomorphisms of graded sheaves
\begin{equation}\label{eq.iso_quozienti_Gp} \frac{\sG_p^*}{\sG^*_{p+1}}[p]\cong\bigwedge^p\left({\sL}/{\sA}\right)^\vee\otimes \Omega^*(\sA).\end{equation}
In particular,  $\sG_p^i\not=0$  only for pairs $(i,p)$ such that 
$p\le i\le \rank \sL$ and $p\le \rank\sL-\rank\sA$. For instance, whenever $i=2$ we have $\sG_0^2=\Omega^2(\sL)$, $\sG_3^2=0$,  
\[\begin{split} \sG_1^2&=\{\phi\in \Omega^2(\sL)\mid \phi(a,b)=0\;\forall a,b\in \sA\},\\
\sG_2^2&=\{\phi\in \Omega^2(\sL)\mid \phi(a,l)=0\;\forall a\in \sA,\, l\in \sL\}.\end{split}
\]

Recall that $\HH^*(\sL) = \HH^{*}(X,\Omega^*(\sL))$ denotes the Lie algebroid cohomology of $\sL$, as in Definition~\ref{def.Lie_alg_coh}.

\begin{definition}\label{def.Leray_filtration} In the above notation, the filtration $\Omega^*(\sL)=\sG_0^*\supset \sG_1^*\cdots$ is called 
the \emph{Leray filtration} of the Lie pair $(\sL,\sA)$. We shall call the associated spectral sequence in hypercohomology
\[ E_1^{p,q}=\HH^{q}\left(X,\sG^*_p/\sG^*_{p+1}[p]\right)\Rightarrow \HH^{p+q}(\sL)\]
the \emph{Leray spectral sequence} of the Lie pair $(\sL,\sA)$. 
\end{definition}

The name \emph{Leray filtration} is motivated by Example~\ref{ex.holomorphicLeray} below. Notice however that for the Lie pair 
$(\Theta_X,0)$ the Leray filtration coincides with the Hodge filtration on differential forms.
\medskip

Given an $\sA$-module $(\sE,\nabla)$, we can also define a filtration $\sG^*_r(\sE)=\sG^*_r\otimes \sE$ of the graded sheaf $\Omega^*(\sL,\sE)$; equivalently,  $\sG^*_r(\sE)$ may be defined as the image of the multiplication map
\[ \sG^*_r\otimes \Omega^*(\sL,\sE)\to \Omega^*(\sL,\sE).\]
If $\nabla'$ is an $\sL$-connection on $\sE$ extending $\nabla$, then by Leibniz rule the filtration 
$\sG^*_r(\sE)$ is preserved by $\nabla'$ and we can immediately see that the maps induced on 
the quotients $\sG^*_r(\sE)/\sG^*_{r+1}(\sE)$ are independent of $\nabla'$ and square-zero operators. Notice also that the curvature of $\nabla'$ belongs to $\sG^2_2(\END_{\Oh_X}(\sE))$ if and only if $[\nabla'_l,\nabla'_a]=\nabla'_{[l,a]}$ for every $l\in \sL$ and $a\in \sA$.

Since $\nabla$ always admits extensions locally (see Remark~\ref{rem.estensione-locale} below), for every $r$ there is a properly defined  structure of differential 
graded sheaf  on $\sG^*_r(\sE)/\sG^*_{r+1}(\sE)$.

It is interesting to point out that the groups  
$E_1^{p,q}=\HH^{q}\left(X,\sG^*_p/\sG^*_{p+1}[p]\right)$, and more generally the hypercohomology groups of 
$\sG^*_r(\sE)/\sG^*_{r+1}(\sE)$, are cohomology groups of $\sA$ with coefficients in suitable $\sA$-modules. In fact, 
there is a canonical $\sA$-module structure on the quotient sheaf $\sL/\sA$ given by the \emph{Bott connection}: denoting by $\pi\colon \sL\to\sL/\sA$ the projection, the connection is defined by the formula
\[ \nabla^B_a \pi(b) = \pi([a,b]),\qquad\forall\,a\in\sA,\; b\in\sL.\]

Therefore, there is a canonical $\sA$-module structure on $\bigwedge^r\left({\sL}/{\sA}\right)^\vee$ for every $r$.

\begin{lemma}\label{lem.bott0} Let $(\sL,\sA)$ be a Lie pair and let $\sE$ be an $\sA$-module. Then
for every $r \geq 1$, the differential graded sheaf  
$\dfrac{\sG_r^*(\sE)}{\sG^*_{r+1}(\sE)}[r]$ is isomorphic to the standard complex of the 
$\sA$-module $\bigwedge^r\left({\sL}/{\sA}\right)^\vee\otimes \sE$. In particular, the Leray spectral sequence of the pair $(\sL,\sA)$ is 
\[ E_1^{p,q}=\HH^{q}\left(\sA;{\bigwedge}^p\left({\sL}/{\sA}\right)^\vee\right).\]
\end{lemma}

\begin{proof}
	For every $r \geq 1$, 
	consider the isomorphism of graded sheaves 
	$ \varphi \colon \frac{\sG_r^*}{\sG^*_{r+1}}[r] \to \bigwedge^r\left({\sL}/{\sA}\right)^\vee \otimes \Omega^*(\sA)$ of \eqref{eq.iso_quozienti_Gp}.
	We begin by showing that this is an isomorphism of complexes, where  the differential on the left is induced by $d_{\sL}$, and the differential on the right is given by the dual connection to the Bott connection. 
	
		Denote by $\nabla^B$ the Bott connection on $\sL/\sA$, and by $\nabla^{B, \vee}$ the induced connection on $\bigwedge^r(\sL/\sA)^\vee$ for every $r \geq 0$.
	We 
	denote by $a_{\sL}$ and $a_{\sA}$ the anchor maps of $\sL$ and $\sA$ respectively.
	Finally, denote by $j $ the inclusion $ j \colon  \left(\frac{\sL}{\sA}\right)^\vee[-1] \to \Omega^1(\sL)$, and by $\pi$ the projection $\pi \colon \sL \to \frac{\sL}{\sA}$, so that 
	for $m \in \sL$ and $\eta \in \left(\frac{\sL}{\sA}\right)^\vee[-1]$ one has that 
	$  m \contr j(\eta) = (j(\eta))(m)= \eta (\pi(m)) = \pi(m) \contr \eta $.  
	For every $\eta \in {\sG_r^*}/{\sG^*_{r+1}} [r]$, we prove that
	\[ \varphi (d_{\sL} \eta) = \nabla^{B, \vee} \varphi( \eta).\] Firstly, consider $\omega \in {\sG_1^*}/{\sG^*_2} [1]\cong  (\sL/\sA)^\vee \otimes \Omega^*(\sA)$ of degree zero, so that 
	$\omega $ belongs to $\sG^1_1/\sG^1_2 [1]= \sG^1_1 [1] \cong (\sL/\sA)^\vee$. Then $d_{\sL} \omega $ belongs to $\sG^2_1[1]$, but we consider its projection to $\frac{\sG^2_1}{\sG^2_2}[1] \cong \left(\frac{\sL}{\sA}\right)^\vee \otimes \Omega^1(\sA)$.
	Hence we calculate it on $b \in \sA$ and $\pi(l) \in \frac{\sL}{\sA} $, obtaining
	\begin{align*}
	d_{\sL} \omega (b,\pi(l)) &= a_{\sL} (b) (j(\omega) (l)) - a_{\sL}(l)(j(\omega) (b)) - j(\omega )([b,l]) \\
	&=	 a_{\sA} (b) (\omega(\pi (l))) - a_{\sL}(l)(\omega (\pi(b))) - \omega (\pi([b,l]))\\
	& = a_{\sA} (b) (\omega(\pi (l)))  - \omega (\pi([b,l])),
	\end{align*}
	since $\pi(b)=0$.
	The connection $\nabla^{B, \vee}$ for $\omega \in \left(\frac{\sL}{\sA}\right)^\vee$, $b \in \sA$ and $\pi(l) \in  \frac{\sL}{\sA}$ is given by
	\begin{align*}
	\pi(l) \contr \nabla^{B, \vee}_b \omega &= d_{\sL} (\pi(l) \contr \omega) (b) - (\nabla^B_b \pi(l)) \contr \omega  =a_{\sL} (b) (\pi(l) \contr \omega) - (\pi ([b,l])) \contr \omega \\
	&= a_{\sA} (b) ( \omega(\pi(l))) - \omega(\pi ([b,l])),
	\end{align*}
	therefore $d_{\sL} \omega = \nabla^{B, \vee} \omega $.

	Consider now $\eta \in \frac{\sG_r^*}{\sG^*_{r+1}}[r]$ of degree $k - r  \geq 0$, which we can assume to be of the form $\eta = \omega_1 \cdots \omega_k$, with $\omega_i \in \Omega^1(\sL)[1]$ for $i=1, \ldots ,r$ such that $\varrho (\omega_1)= \cdots = \varrho(\omega_r)=0$ (i.e., $\omega_i \in (\sL/\sA)^\vee$ for $i=1, \ldots, r$) and $\omega_j \in \Omega^1(\sL)$ for $j= r+1, \ldots ,k$ such that  $\varrho(\omega_{r+1}), \ldots ,\varrho(\omega_k) \neq 0$. 
	
	Then we have that $$ \varphi \colon \frac{\sG_r^*}{\sG^*_{r+1}}[r] \to \bigwedge^r\left(\frac{\sL}{\sA}\right)^\vee\otimes \Omega^*(\sA), \quad \varphi (\eta)= \omega_1 \cdots \omega_r \otimes \varrho (\omega_{r+1})\cdots \varrho(\omega_k),$$
	and so
	\begin{align*}
	\nabla^{B, \vee} (\varphi (\eta )) &= \nabla^{B, \vee} ( \omega_1 \cdots \omega_r \otimes \varrho (\omega_{r+1})\cdots \varrho(\omega_k) )\\
	&= \sum_{i=1}^{r}  \omega_1 \cdots \nabla^{B, \vee}(\omega_i) \cdots\omega_r \otimes \varrho({\omega_{r+1}}) \cdots \varrho(\omega_k) \\
	&\quad+ \sum_{i=r+1}^k (-1)^{i-1-r} \omega_1 \cdots \omega_r \otimes \varrho (\omega_{r+1}) \cdots d_{\sA}(\varrho(\omega_i)) \cdots \varrho(\omega_k )  \\
	&= \sum_{i=1}^{r}  \omega_1 \cdots \nabla^{B, \vee}(\omega_i) \cdots\omega_r \otimes \varrho({\omega_{r+1}}) \cdots \varrho(\omega_k) \\
	&\quad+ \sum_{i=r+1}^k (-1)^{i-1-r} \omega_1 \cdots \omega_r \otimes \varrho (\omega_{r+1}) \cdots \varrho(d_{\sL}(\omega_i)) \cdots \varrho(\omega_k )\\
	&= \sum_{i=1}^{r}  \omega_1 \cdots d_{\sL}(\omega_i) \cdots\omega_r \otimes \varrho({\omega_{r+1}}) \cdots \varrho(\omega_k) \\
	&\quad+ \sum_{i=r+1}^k (-1)^{i-1-r} \omega_1 \cdots \omega_r \otimes \varrho (\omega_{r+1}) \cdots \varrho(d_{\sL}(\omega_i)) \cdots \varrho(\omega_k )\\
	&= (\Id \otimes \varrho) \left(\sum_{i=1}^r \omega_1 \cdots d_{\sL}(\omega_i) \cdots \omega_k + \sum_{i=r+1}^k (-1)^{i-1 -r} \omega_1 \cdots d_{\sL}(\omega_i) \cdots \omega_k\right)\\
	&= \varphi (d_{\sL}(\omega_1 \cdots \omega_k)) = \varphi(d_{\sL} (\eta)).
	\end{align*}

	For every $r \geq 1$, it  follows by \eqref{eq.iso_quozienti_Gp} and by the definition of $\sG_r^*(\sE)$ that
	 there is an isomorphism of graded sheaves
		\[ \varphi \otimes \Id_{\sE} \colon  \frac{\sG_r^*(\sE)}{\sG^*_{r+1}(\sE)}[r]\to \bigwedge^r\left({\sL}/{\sA}\right)^\vee \otimes \Omega^*(\sA)\otimes \sE.\]
	Denote by $\nabla$ the flat $\sA$-connection on $\sE$, and by $\nabla'$ a local extension of $\nabla$ to an $\sL$-connection on $\sE$, which is such that $(\varrho \otimes \Id) \nabla' = \nabla$ and which induces a differential on 	${\sG_r^*(\sE)}/{\sG^*_{r+1}(\sE)}[r]$.

	Take now $\eta \otimes e \in \sG^*_r (\sE) [r]= (\sG^*_r \otimes \sE)[r]$,
	then
	$\nabla' (\eta \otimes e)= d_{\sL}\eta \otimes e + (-1)^{|\eta|} \eta \otimes \nabla'(e)$, and \begin{align*}
	(\varphi \otimes \Id_{\sE}) (\nabla' (\eta \otimes e))&= \varphi(d_{\sL}\eta) \otimes e + (-1)^{|\eta|} \varphi (\eta) \otimes (\varphi \otimes \Id_{\sE}) \nabla'(e)\\
	&= \nabla^{B, \vee} (\varphi (\eta )) \otimes e+ (-1)^{|\eta|} \varphi (\eta) \otimes (\varphi \otimes \Id_{\sE}) \nabla'(e).
	\end{align*}
	Since \[ (\nabla^{B, \vee} \otimes \nabla )((\varphi \otimes \Id_{\sE})(\eta \otimes e)) = (\nabla^{B, \vee} \otimes \nabla )(\varphi(\eta) \otimes e)=\nabla^{B, \vee} (\varphi (\eta )) \otimes e+ (-1)^{|\eta|} \varphi (\eta) \otimes \nabla(e),\]
	it remains only to show that $(\varphi \otimes \Id_{\sE}) \nabla'(e)= \nabla (e)$ for every $e \in \sE$, which follows by the definition of $\varphi$ and by the fact that $(\varrho \otimes \Id) \nabla' = \nabla$, since $\nabla'$ is a local extension of $\nabla$.
\end{proof}

\begin{example}\label{ex.holomorphicLeray}
Let $f\colon X\to Y$ be a smooth morphism of irreducible smooth schemes. Then a Lie pair on $X$ is given by
$(\Theta_X,\Theta_f)$, where $\Theta_f=\HOM_{\sO_X}(\Omega_{X/Y},\sO_X)$ is the subsheaf of relative vector fields: 
since $f$ is smooth there exists an exact sequence of sheaves
\[ 0\to \Theta_f\to \Theta_X\to f^*\Theta_Y\to 0.\]
In this case $\Omega^*(\sL)=\Omega^*_X$ is the usual de Rham  complex of $X$, while 
$\Omega^*(\sA)=\Omega^*_{X/Y}$ is the relative de Rham complex and the filtration $\sG_r^*$ is the algebraic 
analogue of the holomorphic Leray filtration, see  \cite[17.2]{Voisin}, \cite[2.16]{zucker79}.

Since the relative de Rham differential is $f^{-1}\Oh_Y$-linear and $\sG_1^*$ is the ideal sheaf generated by $f^{-1}\Omega^1_Y$, for every 
$r$ we have a natural isomorphism of differential graded sheaves 
\[  \frac{\sG_r^*}{\sG^*_{r+1}}\cong f^{-1}\Omega^r_Y\otimes_{f^{-1}\Oh_Y} \Omega^*_{X/Y}\]
and therefore the first page of the Leray spectral sequence is
\[E_1^{p}=\HH^{*}\left(X,\sG^*_p/\sG^*_{p+1}\right)= 
\HH^{*}\left(X,f^{-1}\Omega^p_Y\otimes_{f^{-1}\Oh_Y} \Omega^*_{X/Y}\right)=
\HH^{*}\left(Y,\Omega^p_Y\otimes_{\Oh_Y} Rf_*\Omega^*_{X/Y}\right).\]

It is an easy consequence of Deligne's results on Hodge theory that if $X$ and $Y$ are complex projective manifolds, then 
the  Leray spectral sequence of the Lie pair $(\Theta_X,\Theta_f)$ degenerates at $E_1$.  In fact, by Hodge decomposition we have 
\[Rf_*\Omega^*_{X/Y}=\oplus_q R^qf_*\Omega^*_{X/Y}[-q]\simeq\oplus_q \Oh_Y\otimes_{\C}R^qf_*\C[-q],\] and then 
$E_1^p=\oplus_q H^*(Y,\Omega_Y^p\otimes_{\C} R^qf_*\C)[-p-q]$.
Since $R^qf_*\C$ is a local system with real structure and $Y$ is compact K\"{a}hler, according to \cite[2.11]{zucker79} (see also \cite[8.5]{GoMil1}), the cohomology  of   $\Omega_Y^p\otimes_{\C} R^qf_*\C$ is a direct summand of the cohomology of 
$R^qf_*\C$. Since the (topological) Leray spectral sequence of $Rf_*\C$  degenerates at $E_2$ \cite[2.6.2]{Deligne69}, we have that 
$E_1^p$ is a direct summand of $\HH^*(Y,Rf_*\C)=H^*(X,\C)=\HH^*(X,\Omega_X^*)$.

For every locally free sheaf $\sE$ on $Y$ its pull-back $f^*\sE=\sO_X\otimes_{f^{-1}\sO_Y}f^{-1}\sE$ has a natural structure of $\Theta_f$-module with connection 
\[ \nabla_\eta(g\otimes e)=\eta(g)\otimes e.\]
More generally, every $\Theta_f$-module can be interpreted,  as in \cite{BL}, as a locally free sheaf on $X$ 
which is endowed with a connection relative to $f$ that is flat.  
\end{example}

\bigskip
\section{Reduced Atiyah classes}
\label{sec.reducedAt}

For every Lie algebroid $\sL$ and every $\sO_X$-module $\sF$ we define the sheaf of $\sO_X$-modules
\[ 
P(\sL,\sF)=\{(l,\phi)\in \sL\times \HOM_{\K}(\sF,\sF)\mid \phi(fe)=f\phi(e)+a(l)(f)e,\; f\in \sO_X,\; e\in \sF\}.\]
If $\sF$ is coherent  then also $P(\sL,\sF)$ is coherent. This has been proved in \cite[Prop. 5.1]{DMcoppie} in the case $\sL=\Theta_X$, while for the general case it is sufficient to observe that 
$P(\sL,\sF)=P(\Theta_X,\sF)\times_{\Theta_X}\sL$.

Denoting by $p\colon P(\sL,\sF)\to \sL$ the projection on the first factor, we have two exact sequences of (graded) $\sO_X$-modules
\begin{equation}\label{equ.esattaatiyah} 
\begin{split}&0\to \HOM_{\sO_X}(\sF,\sF)\to P(\sL,\sF)\xrightarrow{p}\sL\,,\\
&0\to \Omega^1(\sL)\otimes \HOM_{\sO_X}(\sF,\sF)\to \Omega^1(\sL)\otimes  P(\sL,\sF)\xrightarrow{p}\Omega^1(\sL)\otimes \sL=\HOM_{\Oh_X}(\sL,\sL)[-1],\end{split}
\end{equation}
where the second sequence is obtained by applying the exact functor $\Omega^1(\sL)\otimes -$ to the first. Now and in the sequel, we will consider $\Id_{\sL}$ as a global section of  $\HOM_{\Oh_X}(\sL,\sL)[-1]$, a graded sheaf concentrated in degree $1$.

\begin{lemma}\label{lem.caratt_conn} 
	In the above setup, there exists a natural bijection between the set  
	of $\sL$-connections on $\sF$ and global sections  $D\in \Gamma(X,\Omega^1(\sL)\otimes P(\sL,\sF))$
	such that $p(D)=\Id_\sL$. 
\end{lemma}

\begin{proof}
	Let $l_1,\ldots,l_r$ be a local frame of $\sL$ with dual frame $\phi_1,\ldots,\phi_r\in \Omega^1(\sL)$. 
	Every $\K$-linear morphism $\nabla\colon \sF\to \Omega^1(\sL,\sF)$ can be written locally as 
	$\nabla=\sum_{i=1}^r \phi_i\cdot D_i$, with $D_i\in \HOM_{\K}(\sF,\sF)$. 
	By definition, $\nabla$ is a connection if and only if for every 
	$f\in \sO_X$, $e\in \sF$, and every $i$ we have 
	\[ D_i(fe)=l_i\contr\nabla(fe)=a(l_i)(f)e+fD_i(e)\]
	and this is equivalent to the fact that $\sum_{i=1}^r \phi_i\otimes (l_i,D_i)\in \Omega^1(\sL)\otimes P(\sL,\sF)$.
\end{proof}

\begin{lemma}
	If $\sF$ is a locally free sheaf, then the morphism $p\colon P(\sL,\sF)\to \sL$ is surjective.
\end{lemma}

\begin{proof}
	We show this locally, with a proof similar to \cite[Lemma 3.1]{DMcoppie}. Let $R$ be a $\K$-algebra, let $(L, [-,-], a)$ be a Lie algebroid over $R$ with anchor map $a \colon L \to \Der_{\K}(R,R)$, and let $F$ be a free $R$-module with basis $\{e_i\}$. We set
	\[ P(L, F) = \{ (l,\phi) \in L \times \Hom_{\K}(F,F) \ |\ \phi (re)= r\phi(e) + a(l)(r)e, \ \forall r \in R,\ e \in F\},\]
	and show that the projection $p \colon P(L,F) \to L$ is surjective. For every $x \in L$, consider the derivation $a(x) \in \Der_{\K}(R,R)$, and set
	\[ w(\sum_i r_i e_i):= \sum_i a(x)(r_i)e_i, \quad r_i \in R.\] Then the pair $(x,w)$ belongs to $P(L, F)$.
\end{proof}

Assume now that $\sF$ is a locally free sheaf, so that the morphism 
$p\colon P(\sL,\sF)\to \sL$ is surjective  and we have an exact sequence of locally free graded sheaves of $\sO_X$-modules 
\[ 0\to \Omega^1(\sL)\otimes \HOM_{\sO_X}(\sE,\sE)\to \Omega^1(\sL)\otimes P(\sL,\sE)\xrightarrow{p}\HOM_{\Oh_X}(\sL,\sL)[-1]\to 0.\]

We can rewrite the above short exact sequence of graded sheaves concentrated in degree $1$ as a sequence of sheaves in degree $0$:
\[ 0\to \Omega^1(\sL)[1]\otimes \HOM_{\sO_X}(\sE,\sE)\to \Omega^1(\sL)[1]\otimes P(\sL,\sE)\xrightarrow{p}\HOM_{\Oh_X}(\sL,\sL)\to 0.\]
By Lemma~\ref{lem.caratt_conn}, there exists an $\sL$-connection on $\sE$ if and only if the identity on $\sL$ lifts to a global section of  $\Omega^1(\sL)[1]\otimes P(\sL,\sE)$. Writing
\[ \At_{\sL}(\sE)=\de(\Id_{\sL})\in H^1(X,\Omega^1(\sL)[1]\otimes \HOM_{\sO_X}(\sE,\sE))=
\Ext^1_X(\sL\otimes\sE,\sE),\]
where $\de$ is the connecting morphism in the cohomology long exact sequence, 
we have that $\At_{\sL}(\sE)=0$ if and only if there exists an $\sL$-connection on $\sE$.

Equivalently, we can define $\At_{\sL}(\sE)$ as the extension class of the short exact sequence 
\[ 0\to \Omega^1(\sL)\otimes \HOM_{\sO_X}(\sE,\sE)\to Q(\sL,\sE)\xrightarrow{p}\sO_X[-1]\to 0.\]
where, by definition, $Q(\sL,\sE)=p^{-1}(\sO_X[-1]\cdot \Id_{\sL})$.
More explicitly, in a local frame $l_1,\ldots,l_r$ of $\sL$, with dual frame $\phi_1,\ldots,\phi_r\in \Omega^1(\sL)$, 
the elements of $Q(\sL,\sE)$ are those of $\Omega^1(\sL)\otimes P(\sL,\sE)$ of  the form $\sum_{i=1}^r \phi_i\otimes (fl_i,D_i)$ for some $f\in \sO_X$.

\bigskip

Let now $(\sL,\sA)$ be a Lie pair on $X$.
Given an $\sA$-connection $\nabla\colon \sE\to \Omega^1(\sA,\sE)$ on $\sE$ locally free
it makes sense to ask whether $\nabla$ lifts to an $\sL$-connection or not. We prove  
that the solution to this problem is completely determined by an obstruction 
\[\de(\nabla)\in \Ext^1_X\left(\frac{\sL}{\sA}\otimes\sE,\sE\right)=\Ext^1_X\left(\sE,\sE\otimes \sG_1^1[1]\right).\]
It is possible to prove, by applying the results of 
\cite[Section 3]{DMcoppie} to 
an injective resolution, that the same holds also if $\sE$ is not locally free; however we don't need this result.

The case $\sA=0$ has been already considered. 
Suppose $\sA\not=0$, then  we have  a commutative diagram with exact rows
\[ \xymatrix{0\ar[r]&\Omega^1(\sL)\otimes \HOM_{\sO_X}(\sE,\sE)\ar[r]\ar[d]^{\alpha}&Q(\sL,\sE)\ar[r]^p\ar[d]^\beta
	&\sO_X[-1]\ar[r]\ar@{=}[d]&0\\
	0\ar[r]&\Omega^1(\sA)\otimes \HOM_{\sO_X}(\sE,\sE)\ar[r]&Q(\sA,\sE)\ar[r]^p&\sO_X[-1]\ar[r]&0}\]
where $\alpha,\beta$ are the natural restriction maps.
In a local frame $l_1,\ldots,l_r$ of $\sL$, 
with dual frame $\phi_1,\ldots,\phi_r\in \Omega^1(\sL)$ and such that $l_1,\ldots,l_s$ is a local frame for 
$\sA$, we have 
\[ \alpha(\sum_{i=1}^r \phi_i\otimes g_i)=\sum_{i=1}^s \phi_i\otimes g_i,\quad 
\beta(\sum_{i=1}^r \phi_i\otimes (fl_i,D_i))=\sum_{i=1}^s \phi_i\otimes (fl_i,D_i).\]

Since $\alpha$ and $\beta$ are surjective, by the snake lemma 
we have an exact sequence 
\begin{equation}\label{equ.shortpair} 
0\to \sG^1_1\otimes\HOM_{\sO_X}\left(\sE,\sE\right)\to Q(\sL,\sE) \xrightarrow{\beta} 
Q(\sA,\sE)\to 0,
\end{equation}
since  $\sG^1_1$ is by definition the kernel of the surjective map $\Omega^1(\sL)\to \Omega^1(\sA)$.
For simplicity we can rewrite the above short exact sequence of graded sheaves living in degree $1$ as a short exact sequence of sheaves in degree $0$:
\begin{equation*}
0\to \sG^1_1[1] \otimes \HOM_{\sO_X}\left(\sE,\sE\right)\to Q(\sL,\sE)[1] \xrightarrow{\beta} 
Q(\sA,\sE)[1]\to 0.
\end{equation*}
Then the $\sA$-connection $\nabla$ is an element
of $H^0(Q(\sA,\sE)[1])$ such that $p(\nabla)=1$, and the element
\[ \overline{\At}_{\sL/\sA}(\sE, \nabla):=\de(\nabla)\in  H^1\left(X,\sG^1_1[1] \otimes \HOM_{\sO_X}\left(\sE,\sE\right)\right)=
\Ext^1_X\left(\sE,\sE\otimes\sG^1_1[1]\right),\]
is the obstruction to lifting $\nabla$ to an $\sL$-connection. We will call this the \emph{reduced Atiyah class} of $(\sE, \nabla)$.
 
\bigskip
\section{Simplicial $\sL$-connections}

In this section, following \cite{lepri}, we define simplicial $\sL$-connections for a Lie algebroid $\sL$, and simplicial extensions of an $\sA$-connection for a Lie pair $(\sL,\sA)$.
We prove that the adjoint operator of a simplicial $\sL$-connection on a locally free sheaf $\sE$ induces a curved DG-algebra structure on $\Tot(\sU, \Omega^*(\sL) \otimes \HOM^*_{\Oh_X}(\sE,\sE))$. In the case of a Lie pair $(\sL,\sA)$ and of a simplicial extension of a flat $\sA$-connection $\nabla$ on $\sE$, we obtain the data of a curved DG-pair.
Simplicial connections allow us to give representatives of the classes $\At_{\sL}(\sE)$ and $\overline{\At}_{\sL/\sA}(\sE, \nabla)$, and a representative of the obstruction to extending a flat $\sA$-connection on $\sE$ to a $\sL$-connection on $\sE$ with curvature in $\sG_2^2\otimes\HOM_{\Oh_X}(\sE,\sE)$.

Let $\sL$ be a Lie algebroid on $X$ and $\sE$ a locally free sheaf. We have seen that 
$\sL$-connections on $\sE$ exist locally but in
general it does not exist any globally defined connection. However we can define a weaker notion of
connection, which always exists and equally gives a significative example of curved DG-algebra.

In the notation of Sections~\ref{sec.liealgebroid} and~\ref{sec.reducedAt}, consider the short exact sequence
\begin{equation}\label{eq.L-connessione}
	0\to \Omega^1(\sL)\otimes \HOM_{\sO_X}(\sE,\sE)\to \Omega^1(\sL)\otimes P(\sL,\sE)\xrightarrow{\,p\,}\Omega^1(\sL)\otimes\sL\to 0,\end{equation}
and recall that by Lemma~\ref{lem.caratt_conn} an $\sL$-connection on $\sE$ is a global section $D$ of $\Omega^1(\sL)\otimes P(\sL,\sE)$ such that $p(D)=\Id_{\sL}$, where $\Id_{\sL}$ is considered as a global section of $\Omega^1(\sL)\otimes \sL$.
Fix an affine open cover $\sU= \{ U_i\}$ of $X$; by the exactness of the Thom--Whitney totalisation functor one obtains a short exact sequence of DG-vector spaces
\begin{center}
	\begin{tikzcd}[cramped,column sep=small]
		0 \arrow[r] & \Tot(\sU, \Omega^1(\sL)\otimes \HOM_{\sO_X}(\sE,\sE)) \arrow[r] & \Tot(\sU, \Omega^1(\sL)\otimes P(\sL,\sE)) \arrow[r, "p"] & \Tot(\sU, \Omega^1(\sL)\otimes \sL ) \arrow[r] &0.
	\end{tikzcd}
\end{center}
Because of the natural inclusion \eqref{equ.inclusionsecriontot} of global sections in the totalisation, we can consider $\Id_{\sL}$ as an element of $\Tot(\sU, \Omega^1(\sL)\otimes \sL )$.
\begin{definition}
	A simplicial $\sL$-connection on $\sE$ is a lifting $\nabla$ in $\Tot(\sU, \Omega^1(\sL)\otimes P(\sL,\sE))$ of $\Id_{\sL}$ in $\Tot(\sU, \Omega^1(\sL)\otimes \sL )$.
\end{definition}
It is clear that a simplicial $\sL$-connection on $\sE$ always exists.

\medspace

In the case of a Lie pair $(\sL,\sA)$ and of an $\sA$-connection $\nabla^{\sA}$ on the locally free sheaf $\sE$, we can define an analogous notion of simplicial $\sL$-connection extending $\nabla^{\sA}$. It is not restrictive to assume $\sA\not=0$; then 
the exact sequence of locally free graded sheaves \eqref{equ.shortpair}
\[ 0\to \sG^1_1 \otimes \HOM_{\sO_X}\left(\sE,\sE\right)\to Q(\sL,\sE) \xrightarrow{\beta} 
Q(\sA,\sE)\to 0\]
induces
the short exact sequence of DG-vector spaces
\begin{equation}\label{eq.SuccEs_TotPairs}
	0\to \Tot\left(\sU, \sG^1_1 \otimes \HOM_{\sO_X}\left(\sE,\sE\right)\right)\to \Tot(\sU,Q(\sL,\sE))\xrightarrow{\beta}
	\Tot(\sU,Q(\sA,\sE))\to 0.
\end{equation}

We have already observed that an $\sA$-connection $\nabla^{\sA}$ on $\sE$ is a global section of $Q(\sA,\sE)$ such that $p (\nabla^{\sA})=1$, where $p \colon Q(\sA,\sE) \to \sO_X[-1]$ is induced by the map $p$ of \eqref{eq.L-connessione}. By the inclusion of global sections in the totalisation, $\nabla^{\sA}$ belongs to $ \Tot(\sU, Q(\sA,\sE))$.
\begin{definition}\label{def.simpl_ext_conn}
	By a simplicial extension of an $\sA$-connection $\nabla^{\sA}$ on $\sE$ we mean a lifting $\nabla$ in $\Tot(\sU, Q(\sL, \sE))$ of ${\nabla^{\sA}}$ in $ \Tot(\sU, Q(\sA,\sE))$.
\end{definition}

\begin{remark}\label{rem.estensione-locale}
	Notice that the exact sequence \eqref{equ.shortpair} implies that a local extension of an $\sA$-connection to an $\sL$-connection always exists.
\end{remark}

Since maps on the totalisation are induced locally, a similar argument to that of Lemma~\ref{lem.bott0} shows that every simplicial extension $\nabla'$ of a flat $\sA$-connection $\nabla^{\sA}$ on $\sE$ induces a differential on the complex $\Tot(\sU,\sG_r^*(\sE)/\sG_{r+1}^*(\sE)[r])$. We then have that 
$H^*(\Tot(\sU,\sG_r^*(\sE)/\sG_{r+1}^*(\sE)[r])) \cong \HH^*(X,\sG_r^*(\sE)/\sG_{r+1}^*(\sE)[r]) $ 
is isomorphic to the Lie algebroid cohomology of $\sA$ with coefficients in the $\sA$-module $\bigwedge^r(\sL/\sA)^\vee\otimes \sE$, again by Lemma~\ref{lem.bott0}.

\begin{lemma}\label{lem.ostruzioni_connessioni}
	For a Lie algebroid $\sL$ and a simplicial $\sL$-connection $\nabla$ on $\sE$, the cohomology class of $d_{\Tot}\nabla$ in $\Tot(\sU, \Omega^1(\sL) \otimes \HOM_{\Oh_X}(\sE,\sE))$ is the obstruction $\At_{\sL}(\sE)$ to the existence of an $\sL$-connection on $\sE$.
	
	For a Lie pair $(\sL, \sA)$ and a simplicial extension $\nabla$ of an $\sA$-connection $\nabla^{\sA}$ on $\sE$, the cohomology class of $d_{\Tot}\nabla$ in 
	$\Tot(\sU, \sG^1_1 \otimes \HOM_{\Oh_X}(\sE, \sE))$ is the obstruction  $ \overline{\At}_{\sL/\sA}(\sE, \nabla^{\sA})$ to the extension of $\nabla^{\sA}$ to an $\sL$-connection.
\end{lemma}

\begin{proof} According to Example~\ref{ex.cech} we have natural isomorphisms
\[ \begin{split}
H^0(\Tot(\sU, \Omega^1(\sL)\otimes P(\sL,\sE)))&=\Gamma(X,\Omega^1(\sL)\otimes P(\sL,\sE)),\\
H^0(\Tot(\sU, \Omega^1(\sL)\otimes \HOM_{\Oh_X}(\sE,\sE)))&=\Gamma(X,\Omega^1(\sL)\otimes \HOM_{\Oh_X}(\sE,\sE)).\end{split}
\]

	Consider first the case of a simplicial $\sL$-connection $\nabla$ on $\sE$;
notice that $d_{\Tot}\nabla$ belongs to $\Tot(\sU, \Omega^1(\sL) \otimes \HOM_{\Oh_X}(\sE,\sE))$, because $p (d_{\Tot}\nabla) = d_{\Tot} p (\nabla)= d_{\Tot} \Id_{\sL}=0$, since 
	$\Id_{\sL}$ is a global section.  If there exists an $\sL$-connection $\nabla'$ on $\sE$ it belongs to $\Tot(\sU, \Omega^1(\sL)\otimes P(\sL,\sE))$ by the inclusion of global sections in the totalisation, and one has that $d_{\Tot} \nabla'=0$. Then for any simplicial connection $\nabla$, the difference $\nabla -\nabla'$ belongs to
	$\Tot(\sU,  \Omega^1(\sL)\otimes \HOM_{\Oh_X}(\sE,\sE))$ and $d_{\Tot}  (\nabla -\nabla')= d_{\Tot}\nabla$, so that $d_{\Tot}\nabla $ is trivial in the cohomology of $\Tot(\sU,  \Omega^1(\sL)\otimes \HOM_{\Oh_X}(\sE,\sE))$. Conversely, if $d_{\Tot}\nabla = d_{\Tot}\varphi$, with $\varphi \in \Tot(\sU,  \Omega^1(\sL)\otimes \HOM_{\Oh_X}(\sE,\sE))$, then $\nabla - \varphi$ is a global $\sL$-connection on $\sE$.

	In the case of a Lie pair $(\sL, \sA)$ and a  simplicial extension $\nabla$ of an $\sA$-connection $ \nabla^{\sA}$ on $\sE$, notice that $d_{\Tot}\nabla$ belongs to $\Tot\left(\sU, \sG^1_1 \otimes \HOM_{\sO_X}\left(\sE,\sE\right)\right)$: in fact, $\beta (d_{\Tot}\nabla)= d_{\Tot}\beta(\nabla)= d_{\Tot} \nabla^{\sA}=0$, because $\nabla^{\sA}$ is a global section.
	If $\nabla^{\sA}$ extends to an $\sL$-connection there exists $\nabla'$ in $\Gamma(X,  Q(\sL,\sE))$ with $\beta(\nabla')= \nabla^{\sA}$, which is such that $d_{\Tot}\nabla' =0$ in $\Tot(\sU, Q(\sL,\sE))$, because it is a global section. Then for every simplicial connection $\nabla$ lifting $\nabla^{\sA}$, $\nabla - \nabla'$ belongs to the kernel of $\beta$, which is 
	$\Tot(\sU, \sG_1^1 \otimes \HOM_{\sO_X}\left(\sE,\sE\right))$, and
	$ d_{\Tot}(\nabla - \nabla')= d_{\Tot}\nabla,$
	so that $d_{\Tot}\nabla$ is trivial in cohomology.  Vice versa, if $d_{\Tot}\nabla= d_{\Tot}\phi$ is trivial in the cohomology of $\Tot(\sU, \sG^1_1 \otimes \HOM_{\Oh_X}(\sE, \sE))$, it is easy to see that $\nabla - \phi$ is a connection lifting $\nabla^{\sA}$.
\end{proof}

A simplicial $\sL$-connection on a locally free sheaf $\sE$ induces a curved DG-algebra structure on the DG-vector space $\Tot(\sU, \Omega^*(\sL, \HOM_{\sO_X}(\sE,\sE)))$. To see this, the first step is the construction of an adjoint operator for the simplicial connection, which is done via the following lemma.

\begin{lemma}\label{lem.bracket} In the above situation, the $\Oh_X$-bilinear map 
	\[ \begin{split}
	[-,-]& \colon (\Omega^1(\sL)\otimes P(\sL,\sE)) \times \HOM_{\sO_X}(\sE,\sE) \to \Omega^1(\sL)\otimes \HOM_{\sO_X}(\sE,\sE),\\[5pt]
[\eta \otimes (l,v), g]&= \eta \otimes [v,g],\qquad \eta \in \Omega^1(\sL),\, (l,v) \in P(\sL, \sE),\, g \in \HOM_{\Oh_X}(\sE,\sE).\end{split}
\]
is well defined.	
\end{lemma}

\begin{proof}
	For $r \in \Oh_X$, 
	\begin{align*} [v,g](re) &= v (r g(e)) - g (r v(e) + a (l) (r)e)= rvg(e) + a(l)(r)g(e)- rgv(e)- a (l)(r)g(e)\\
		&= r [v,g](e),
	\end{align*}
	so $[v,g]$ belongs to $\HOM_{\Oh_X}(\sE,\sE)$.
	The bracket is well-defined: for $r \in \Oh_X$,
	\begin{align*}
		&[\eta \otimes (rl, rv),g] = \eta \otimes [rv, g]= \eta \otimes r[v,g]=r \eta \otimes [v,g]=[r\eta \otimes (l,v), g ].
	\end{align*}
\end{proof}

The bracket defined in Lemma~\ref{lem.bracket} induces a graded Lie bracket on the totalisation
\[ [-,-] \colon \Tot(\sU, \Omega^1(\sL)\otimes P(\sL,\sE)) \times \Tot(\sU, \HOM_{\Oh_X}(\sE,\sE)) \to \Tot(\sU,\Omega^1(\sL)\otimes \HOM_{\Oh_X}(\sE,\sE)),\]
which allows to define the adjoint operator to a simplicial $\sL$-connection $\nabla$  on $\sE$:
\begin{equation}\label{eq.operatoreaggiunto}
	d_{\nabla} := [\nabla, -] \colon \Tot(\sU, \HOM_{\Oh_X}(\sE,\sE)) \to \Tot(\sU,\Omega^1(\sL)\otimes \HOM_{\Oh_X}(\sE,\sE)).
\end{equation}

Recall that since $\Omega^*(\sL,\HOM_{\Oh_X}(\sE,\sE))$ is a sheaf of graded algebras  and the Tot functor preserves multiplicative structures,  $\Tot(\sU,\Omega^*(\sL,\HOM_{\Oh_X}(\sE,\sE)))$ is a differential graded algebra, with differential denoted by $d_{\Tot}$.

\begin{lemma}\label{lem.curvedalgebra}
	The adjoint operator 
	\[d_{\nabla} = [\nabla, -] \colon \Tot(\sU, \HOM_{\Oh_X}(\sE,\sE)) \to \Tot(\sU,\Omega^1(\sL)\otimes \HOM_{\Oh_X}(\sE,\sE))\] 
	extends for every $i \geq 0$ to a $\K$-linear operator 
	$$d_{\nabla} \colon \Tot(\sU,\Omega^i(\sL)\otimes \HOM_{\Oh_X}(\sE, \sE)) \to \Tot(\sU, \Omega^{i+1}(\sL)\otimes \HOM_{\Oh_X}(\sE, \sE) ).$$	
	Then $(\Tot(\sU,\Omega^*(\sL, \HOM_{\Oh_X}(\sE, \sE))), d_{\Tot} + d_\nabla)$ is a curved DG-algebra with curvature
	\[  d_{\Tot} \nabla + C,\]
	with $d_{\Tot}\nabla \in \Tot(\sU, \Omega^1(\sL)\otimes \HOM_{\Oh_X}(\sE,\sE))$ and $C \in \Tot(\sU,\Omega^2(\sL)\otimes \HOM_{\Oh_X}(\sE, \sE))$  such that $d_{\nabla}^2 = [C, -]$.
\end{lemma}

\begin{proof}
	Consider first the case of a germ of an $\sL$-connection, i.e., an element $Y$ of $\Gamma(V,\Omega^1(\sL) \otimes P(\sL,\sE))$ such that $p(Y)= \Id_{\sL}|_V$, for some open set $V \subset X$. 
	As usual, $Y$ extends uniquely to a $\K$-linear morphism of degree 1
	\[ Y \colon \Omega^*(\sL,\sE)|_V \to \Omega^*(\sL,\sE)|_V\] such that
	\begin{equation*}
		Y(\eta \otimes e)= d_{\sL}(\eta)\otimes e+ (-1)^{|\eta |} \eta \otimes Y(e).
	\end{equation*}
	for all $\eta\in\Omega^*(\sL)|_V, e \in \sE|_V$. It is easy to see that the map $Y^2$
	is $\sO_X$-linear,
	so it can be identified with a section  of $\Omega^2(\sL, \HOM_{\sO_X}(\sE,\sE))|_V$.
	
	One can define an adjoint operator
	\[ d_Y:= [Y,-] \colon \HOM_{\Oh_X}(\sE,\sE)|_V \to \Omega^1(\sL, \HOM_{\Oh_X}(\sE,\sE))|_V,\]
	which can be extended for all $i \geq 0$ to an operator 
	$$d_Y  \colon \Omega^i(\sL, \HOM_{\Oh_X}(\sE,\sE))|_V \to \Omega^{i+1}(\sL,\HOM_{\Oh_X}(\sE,\sE))|_V$$ 
	by setting
	\begin{equation}\label{eq.connessione}
		d_Y (\eta \otimes f):= d_{\sL} (\eta )\otimes f + (-1)^{|{\eta}|}\eta \otimes [Y, f],
	\end{equation}
	where $[Y, f]$ denotes the Lie bracket of Lemma~\ref{lem.bracket}. 
	
	As in the classical case, one can see that 
	\begin{equation}\label{eq.quadrato}
		d_Y^2 (\eta \otimes f)= [Y^2, \eta \otimes f]
	\end{equation}for all $\eta \in \Omega^*(\sL)|_V$ and $f \in \HOM_{\Oh_X}(\sE, \sE)|_V$.

	Let now $\nabla$ be a simplicial $\sL$-connection on $\sE$, namely an element of $\Tot(\sU, \Omega^1(\sL) \otimes P(\sL,\sE))$ such that $p(\nabla)= \Id_{\sL} \in \Tot(\sU, \Omega^1(\sL)\otimes \sL)$. 
	Then for every $i \geq 0$ the extension of the  operator $d_\nabla = [\nabla,-]$, defined in \eqref{eq.operatoreaggiunto}, to an operator
	$$d_\nabla = [\nabla, -] \colon  \Tot(\sU,\Omega^i(\sL)\otimes \HOM_{\Oh_X}(\sE, \sE)) \to \Tot(\sU, \Omega^{i+1}(\sL)\otimes \HOM_{\Oh_X}(\sE, \sE) )$$
	can be defined by using the map induced by \eqref{eq.connessione} on the totalisation, and one obtains a degree one operator
	$$d_\nabla = [\nabla, -] \colon  \Tot(\sU,\Omega^*(\sL, \HOM_{\Oh_X}(\sE, \sE))) \to \Tot(\sU, \Omega^{*}(\sL, \HOM_{\Oh_X}(\sE, \sE))).$$
	In detail, let $\nabla= (D_n)$ with $D_n \in A_n \otimes \prod_{i_1, \ldots, i_n} (\Omega^1(\sL)\otimes P(\sL,\sE)) (U_{i_1, \ldots, i_n})$ such that $p(D_n)= 1\otimes(\Id_{\sL}|_{U_{i_1, \ldots, i_n}})$ for every $n\geq 0$. Since maps on the totalisation are defined componentwise, it is enough to define the bracket
	\[[D_n, \phi_n \otimes (\omega_{i_1, \ldots, i_n} \otimes f_{i_1, \ldots, i_n})],\]
	for $\phi_n \otimes (\omega_{i_1, \ldots, i_n} \otimes f_{i_1, \ldots, i_n})$ in $A_n \otimes \prod_{i_1, \ldots, i_n} (\Omega^i(\sL)\otimes \HOM_{\Oh_X}(\sE,\sE)(U_{i_1, \ldots, i_n})$. Let  \begin{equation}\label{eq.Dn}
		D_n = \sum_{j} \eta_{j,n} \otimes (t_{j,i_1, \ldots, i_n}),\quad \eta_{j,n} \in A_n,\quad t_{j,i_1, \ldots, i_n} \in (\Omega^1(\sL)\otimes P(\sL,\sE))(U_{i_1, \ldots, i_n});
	\end{equation}then the bracket can be  defined as
	\begin{align*}
		&[D_n, \phi_n \otimes (\omega_{i_1, \ldots, i_n} \otimes f_{i_1, \ldots, i_n})]= [\sum_{j} \eta_{j,n} \otimes (t_{j,i_1, \ldots, i_n}), \phi_n \otimes (\omega_{i_1, \ldots, i_n} \otimes f_{i_1, \ldots, i_n})]\\
		&=p(D_n) ( \phi_n \otimes (\omega_{i_1, \ldots, i_n} \otimes f_{i_1, \ldots, i_n}))\\
		&\quad + (-1)^{|{\phi_n}| + |{\omega_{i_1,\ldots, i_n}}|}  \sum_j \phi_n \eta_{j,n} \otimes (\omega_{i_1, \ldots, i_n} \otimes [t_{j,i_1, \ldots, i_n}, f_{i_1, \ldots, i_n}])\\
		&=(1\otimes(\Id_{\sL}|_{U_{i_1, \ldots, i_n}})) ( \phi_n \otimes (\omega_{i_1, \ldots, i_n} \otimes f_{i_1, \ldots, i_n})) \\
		&\quad +(-1)^{|{\phi_n}| + |{\omega_{i_1,\ldots, i_n}}|}  \sum_j \phi_n \eta_{j,n} \otimes (\omega_{i_1, \ldots, i_n} \otimes [t_{j,i_1, \ldots, i_n}, f_{i_1, \ldots, i_n}])\\
		&=(-1)^{|{\phi_n}|}\phi_n \otimes (d_{\sL}\omega_{i_1, \ldots, i_n} \otimes f_{i_1, \ldots, i_n}) \\
		&\quad  +(-1)^{|{\phi_n}| + |{\omega_{i_1,\ldots, i_n}}|}  \sum_j \phi_n \eta_{j,n} \otimes (\omega_{i_1, \ldots, i_n} \otimes [t_{j,i_1, \ldots, i_n}, f_{i_1, \ldots, i_n}]), 
	\end{align*} 
where the bracket $[t_{j,i_1, \ldots, i_n}, f_{i_1, \ldots, i_n}]$ is induced by the one of Lemma~\ref{lem.bracket}.

	For every $i \geq 0$ the simplicial $\sL$-connection $\nabla$ also induces a map
	$$\nabla \colon \Tot(\sU, \Omega^i(\sL)\otimes \sE) \to\Tot(\sU, \Omega^{i+1}(\sL)\otimes \sE) $$
	which allows to define a degree one operator $\nabla\colon \Tot(\sU, \Omega^*(\sL,  \sE)) \to \Tot(\sU, \Omega^*(\sL,  \sE))$.
	In fact, let $\nabla= (D_n)$ as in \eqref{eq.Dn}, and consider $\phi_n \otimes (\omega_{i_1, \ldots, i_n} \otimes e_{i_1, \ldots, i_n})$ in $A_n \otimes \prod_{i_1, \ldots, i_n} (\Omega^1(\sL)\otimes \sE)(U_{i_1, \ldots, i_n})$. Then the operator can be defined as
	\[\begin{split}
		D_n (\phi_n \otimes (\omega_{i_1, \ldots, i_n} \otimes e_{i_1, \ldots, i_n}))&= (\sum_{j} \eta_{j,n} \otimes (t_{j,i_1, \ldots, i_n})) (\phi_n \otimes (\omega_{i_1, \ldots, i_n} \otimes e_{i_1, \ldots, i_n}))\\
			&= p(D_n) (\phi_n \otimes (\omega_{i_1, \ldots, i_n} \otimes e_{i_1, \ldots, i_n}))+\\
		&\quad  +(-1)^{ |{\phi_n}|+ |{\omega_{i_1,\ldots, i_n}}|}  \sum_j \phi_n \eta_{j,n} \otimes (\omega_{i_1 \ldots i_n} \otimes t_{i_1, \ldots, i_n}(e_{i_1, \ldots, i_n}))\\
			&= (1\otimes(\Id_{\sL}|_{U_{i_1, \ldots, i_n}})) (\phi_n \otimes (\omega_{i_1, \ldots, i_n} \otimes e_{i_1, \ldots, i_n}))+\\
		&\quad  +(-1)^{ |{\phi_n}|+ |{\omega_{i_1,\ldots, i_n}}|}  \sum_j \phi_n \eta_{j,n} \otimes (\omega_{i_1 \ldots i_n} \otimes t_{i_1, \ldots, i_n}(e_{i_1, \ldots, i_n}))\\
		&= (-1)^{|{\phi_n}|} \big(\phi_n \otimes (d_{\sL}\omega_{i_1, \ldots, i_n} \otimes e_{i_1, \ldots, i_n})\\
		&\quad  +(-1)^{ |{\omega_{i_1,\ldots, i_n}}|}  \sum_j \phi_n \eta_{j,n} \otimes (\omega_{i_1 \ldots i_n} \otimes t_{i_1, \ldots, i_n}(e_{i_1, \ldots, i_n}))\big).
	\end{split}\]

	Since all the maps considered on the totalisation are induced by the ones defined locally on the complexes of sheaves, for  $d_\nabla = [\nabla, -]$  one has that, by \eqref{eq.quadrato},
	\[ d_\nabla^2 = [C, -], \quad C \in \Tot(\sU, \Omega^2(\sL,\HOM_{\Oh_X}(\sE,\sE))).\]
	Then $d_{\Tot}+ d_{\nabla} $ is a degree one derivation of $\Tot(\sU, \Omega^*(\sL)\otimes \HOM_{\Oh_X}(\sE,\sE))$, with square
	\[ (d_{\Tot}+ d_\nabla)^2= d_{\Tot}^2 + d_{\Tot}[\nabla,-] +  [\nabla,d_{\Tot}-] + d_{\nabla}^2 = [d_{\Tot}\nabla, - ]+ [C,-]= [d_{\Tot}\nabla + C,-],\]
	so the curvature is $d_{\Tot}\nabla + C$.	We have already seen in Lemma~\ref{lem.ostruzioni_connessioni} that $d_{\Tot}\nabla$ belongs to $\Tot(\sU, \Omega^1(\sL) \otimes \HOM_{\Oh_X}(\sE,\sE))$.

	The last thing to prove is that $(d_{\Tot}+ d_{\nabla})(d_{\Tot}\nabla + C)=0$. One has that
	\[(d_{\Tot}+ d_{\nabla})(d_{\Tot}\nabla + C)= d_{\Tot}^2 \nabla + d_{\nabla} d_{\Tot}\nabla + d_{\Tot}C + d_{\nabla}C = d_{\nabla} d_{\Tot}\nabla + d_{\Tot}C. \]
	Then 
	\begin{align*}
		&d_{\nabla} d_{\Tot}\nabla = [\nabla, d_{\Tot} \nabla] = - [d_{\Tot}\nabla , \nabla] = - \frac{1}{2} d_{\Tot}[\nabla, \nabla] = - d_{\Tot}C,
	\end{align*}
	so that $(d_{\Tot}+ d_{\nabla})(d_{\Tot}\nabla + C)=0$.
\end{proof}

\medspace

In the case of a Lie pair $(\sL,\sA)$ and a locally free sheaf $\sE$, the  natural surjective restriction maps
\[ \varrho \colon \Omega^*(\sL)\to\Omega^*(\sA),\qquad \varrho \otimes \Id \colon \Omega^*(\sL,\HOM_{\sO_X}(\sE,\sE))\to\Omega^*(\sA,\HOM_{\sO_X}(\sE,\sE)),\]
induce morphisms on the totalisation
\[ \varrho \colon \Tot(\sU,\Omega^*(\sL))\to\Tot(\sU,\Omega^*(\sA)),\] \[\varrho \otimes \Id\colon\Tot(\sU,\Omega^*(\sL,\HOM_{\sO_X}(\sE,\sE)))\to\Tot(\sU,\Omega^*(\sA,\HOM_{\sO_X}(\sE,\sE))),\]
whose kernels define bilateral ideals
\[ \Tot(\sU,\sG^*_1)=\ker(\varrho)\subset\Tot(\sU,\Omega^*(\sL)),\]
\[\Tot(\sU,\sG^*_1 \otimes \HOM_{\sO_X}(\sE,\sE))=\ker(\varrho \otimes \Id)\subset\Tot(\sU,\Omega^*(\sL,\HOM_{\sO_X}(\sE,\sE))).\]

\begin{lemma}\label{lem.idealecurvato}
	Let $(\sE,\nabla^{\sA})$ be a locally free $\sA$-module, and let $\nabla$ be a simplicial extension of $\nabla^{\sA}$ to an $\sL$-connection. Then $I:=\Tot(\sU,\sG^*_1\otimes \HOM_{\sO_X}(\sE,\sE))$ is a curved ideal of the curved DG-algebra $$(\Tot(\sU,\Omega^*(\sL, \HOM_{\Oh_X}(\sE, \sE))), d_{\Tot} + d_{\nabla}, d_{\Tot}\nabla+ C),$$
	where $C$, the curvature of the simplicial connection $\nabla$, belongs to $\Tot(\sU,\sG^2_1 \otimes \HOM_{\sO_X}(\sE,\sE))$ and $d_{\Tot} \nabla $ belongs to $\Tot\left(\sU,\sG^1_1 \otimes \HOM_{\sO_X}\left(\sE,\sE\right)\right)$.
\end{lemma}
\begin{proof}
	It is clear that  the ideal $I= \Tot(\sU,\sG^*_1\otimes \HOM_{\sO_X}(\sE,\sE))$
	is  $d_{\Tot}$-closed. Let $x$ be an element of $I$, so that $(\varrho \otimes \Id )(x)=0$, then
	\[ (\varrho \otimes \Id) (d_{\nabla}x) =d_{\nabla^{\sA}} (\varrho \otimes \Id) (x) =0,\]
	so $I$ is also $ d_{\nabla}$-closed.
	Since the $\sA$-connection $\nabla^{\sA}$ is flat, the curvature   $C$  of $\nabla$ belongs to $\Tot(\sU,\sG^2_1 \otimes \HOM_{\sO_X}(\sE,\sE)) \subset I$, which is the kernel of the surjective map
	\[ \varrho \otimes \Id \colon \Tot(\sU, \Omega^2(\sL) \otimes \HOM_{\Oh_X}(\sE,\sE)) \to \Tot(\sU, \Omega^2(\sA) \otimes \HOM_{\Oh_X}(\sE,\sE)).
	\] 
	By Lemma~\ref{lem.ostruzioni_connessioni}, 
	$d_{\Tot} \nabla$ belongs to $\Tot\left(\sU,\sG^1_1 \otimes \HOM_{\sO_X}\left(\sE,\sE\right)\right)$,
	therefore it belongs to the ideal $I$. 
\end{proof}

For the ideal $I= \Tot(\sU, \sG^*_1 \otimes \HOM_{\Oh_X}(\sE,\sE))$ we have that
\begin{equation}\label{equ.potenzeideale}
I^{(n)}= \Tot(\sU, \sG^*_n \otimes \HOM_{\Oh_X}(\sE,\sE)).\end{equation}
In fact, the inclusion $I^{(n)} \subset \Tot(\sU, \sG^*_n \otimes \HOM_{\Oh_X}(\sE,\sE))$ is clear. For the other one, it suffices to notice that the multiplication map 
$\underbrace{\sG^*_1 \otimes \cdots \otimes \sG^*_1 }_n\to \sG_n^*$ is surjective on all affine open sets.

According to Definition~\ref{def.atiyahclass}, the Atiyah cocycle of the curved DG-pair 
\[ \left(A= \Tot(\sU,\Omega^*(\sL, \HOM_{\Oh_X}(\sE, \sE))), I= \Tot(\sU,\sG^*_1\otimes \HOM_{\sO_X}(\sE,\sE)) \right)\]
is the class of the curvature $R = d_{\Tot}\nabla + C$ in \[\frac{I}{I^{(2)}} = \Tot\left(\sU, \frac{\sG^*_1}{\sG^*_2}\otimes \HOM_{\sO_X}(\sE,\sE)\right).\]

\begin{theorem}\label{thm.ostruzione_connessione_piatta}
	Given a Lie pair $(\sL,\sA)$ and a locally free $\sA$-module $(\sE, \nabla^{\sA})$,
	the Atiyah class  $\At(A,I)$ of the curved DG-pair 	\[ (A=\Tot(\sU,\Omega^*(\sL, \HOM_{\Oh_X}(\sE, \sE))), I= \Tot(\sU,\sG^*_1\otimes \HOM_{\sO_X}(\sE,\sE)) )\]
	does not depend on the choice of the simplicial $\sL$-connection extending $\nabla^{\sA}$. Moreover, it is the obstruction to the existence of a $\sL$-connection on $\sE$ extending $\nabla^{\sA}$ with curvature in $ \sG_2^2\otimes\HOM_{\Oh_X}(\sE,\sE)$. 
\end{theorem}

\begin{proof}
The difference of two simplicial extensions  $\nabla$ and $\nabla'$ of the $\sA$-connection $\nabla^{\sA}$ belongs to the ideal $I$. In fact, 
	  considering the short exact sequence \eqref{eq.SuccEs_TotPairs},
	\begin{equation*}
		0\to \Tot\left(\sU, \sG^1_1 \otimes \HOM_{\sO_X}\left(\sE,\sE\right)\right)\to \Tot(\sU,Q(\sL,\sE))\xrightarrow{\beta}
		\Tot(\sU,Q(\sA,\sE))\to 0,
	\end{equation*}
	we have that $\beta (\nabla -\nabla')= \nabla^{\sA}-\nabla^{\sA}=0$ and therefore, writing  $\phi := \nabla - \nabla' $, we have   $\phi\in \Tot\left(\sU, \sG^1_1 \otimes \HOM_{\sO_X}\left(\sE,\sE\right)\right) \subset I$.
	Then $d_{\nabla} = d_{\nabla'} + [\phi, -]$ and	the first claim follows from Lemma~\ref{lem.curvatura_astratta}.

	Next, we show that the Atiyah class $\At(A,I)$ of the curved DG-pair is the obstruction to the existence of a $\sL$-connection on $\sE$ extending $\nabla^{\sA}$,  with curvature in $\sG_2^2\otimes\HOM_{\Oh_X}(\sE,\sE)$. By Lemma~\ref{lem.curvatura_astratta}, $\At(A,I)$ is the obstruction to existence of $x \in I= \Tot(\sU,\sG^*_1\otimes \HOM_{\Oh_X}(\sE,\sE))$ of degree 1 such that $ R + (d_{\Tot}+ d_{\nabla})x  $ belongs to $I^{(2)}= \Tot(\sU, \sG^*_2 \otimes \HOM_{\Oh_X}(\sE,\sE))$. Assume that there exists such $x$, and notice that by degree reasons it belongs to  $\Tot(\sU,\sG^1_1\otimes \HOM_{\Oh_X}(\sE,\sE))$, since $\sG^0_1=0$.
	Then, since $\sG^1_2=0$,
	\[ \begin{split}
		&d_{\Tot} \nabla + d_{\Tot}x \ \in I^{(2)} \cap \Tot(\sU,\sG^1_1\otimes \HOM_{\Oh_X}(\sE,\sE))=0\\
		&C + d_{\nabla} x \ \in I^{(2)} \cap \Tot(\sU,\sG^2_1\otimes \HOM_{\Oh_X}(\sE,\sE))= \Tot(\sU,\sG^2_2\otimes \HOM_{\Oh_X}(\sE,\sE)),
	\end{split}\]
	and by the first equation $\nabla + x$ is a global $\sL$-connection on $\sE$ extending $\nabla^{\sA}$.
	
		We denote by  $R_x= d_{\Tot}(\nabla + x) +C_x =C_x$ the curvature of the curved DG-algebra $(A, d_{\Tot} + d_{\nabla +x})$. Then
		\[ R_x = R + (d_{\Tot}+ d_{\nabla})x + \dfrac{1}{2}[x,x] = d_{\Tot}\nabla + C+ d_{\Tot}x+ d_{\nabla}x + \dfrac{1}{2}[x,x] =  C+ d_{\nabla}x + \dfrac{1}{2}[x,x] , \]
		so that the curvature of $\nabla + x$ is equal to $C_x = C + d_{\nabla}x + \frac{1}{2}[x,x]$, which belongs to
$\Tot(\sU, \sG^2_2 \otimes \HOM_{\Oh_X}(\sE,\sE))$. Finally, since $d_{\nabla +x} (C_x) =0$, one has that 
\[ 0 = (d_{\Tot} + d_{\nabla +x})(R_x)=  (d_{\Tot} + d_{\nabla +x})(C_x) = d_{\Tot} C_x,\]
and $C_x$ is a global section of $\sG_2^2\otimes\HOM_{\Oh_X}(\sE,\sE)$.

	The converse is clear.
\end{proof}

By the above, the Atiyah class $\At(A,I)$ of the curved DG-pair \[ (A=\Tot(\sU,\Omega^*(\sL, \END_{\Oh_X}(\sE))), I= \Tot(\sU,\sG^*_1\otimes \END_{\Oh_X}(\sE)) )\] is well-defined: 
\[ \At(A,I) \in \HH^2 \left(X, \dfrac{\sG^*_1}{\sG^*_2} \otimes \END_{\Oh_X}(\sE)\right).\]

\begin{definition} In the above situation, via the isomorphisms of Lemma~\ref{lem.bott0}, we call 
\[\At_{\sL/\sA}(\sE):= \At(A,I) \in \HH^1 \left(\sA;(\sL/\sA)^\vee \otimes \END_{\Oh_X}(\sE)\right).\]
the $(\sL,\sA)$-Atiyah class of $\sE$.
\end{definition}

\begin{remark}
	Recalling that $\sG^1_2=0$, 
	the morphism of graded sheaves $ t  \colon \frac{\sG^*_1}{\sG^*_2} \to \sG^{1}_1$ with kernel $\frac{\sG_1^{\geq 2}}{\sG_2^*}$ induces a morphism of DG-vector spaces
	\[ t \colon \Tot\left(\sU, \dfrac{\sG^*_1}{\sG^*_2} \otimes \END_{\Oh_X}(\sE) \right) \to \Tot(\sU, \sG^1_1 \otimes \END_{\Oh_X}(\sE)),\]
	which sends the class of $R = d_{\Tot}\nabla + C$ to $d_{\Tot}\nabla$. The reduced Atiyah class $\overline{\At}_{\sL/\sA}(\sE,\nabla^{\sA})$ is then the image of the Atiyah class $\At_{\sL/\sA}(\sE)$ of the curved DG-pair 
	\[\left(A=\Tot(\sU,\Omega^*(\sL, \END_{\Oh_X}(\sE))), I= \Tot(\sU,\sG^*_1\otimes \END_{\Oh_X}(\sE)) \right)\] 
via the map induced by $t$ in hypercohomology
	\begin{align*} t \colon \HH^* \left(X, \dfrac{\sG^*_1}{\sG^*_2} \otimes \END_{\Oh_X}(\sE)\right) &\to \HH^*(X, \sG^1_1 \otimes \END_{\Oh_X}(\sE))\\
		\At_{\sL/\sA}(\sE) &\mapsto \overline{\At}_{\sL/\sA}(\sE,\nabla^{\sA}).
	\end{align*}
In particular if $\At_{\sL/\sA}(\sE)$ is trivial, then so is  $\overline{\At}_{\sL/\sA}(\sE,\nabla^{\sA})$.

If we consider the Lie pair $(\sL,0)$, both the obstructions $\At_{\sL/\sA}(\sE)$ and $ \overline{\At}_{\sL/\sA}(\sE,\nabla^{\sA})$ reduce to the obstruction $\At_{\sL}(\sE)$ to the existence of an $\sL$-connection on $\sE$. 
\end{remark}

\begin{corollary}\label{cor.criteriosvanimento}
	Let $(\sL,\sA)$ be a Lie pair on $X$ such that there exists an $\Oh_X$-linear projection $p \colon \sL \to \sA$ which commutes with anchor maps and with  adjoint Lie actions of $\sA$. Then for every $\sA$-module $\sE$ the Atiyah class $\At_{\sL/\sA}(\sE)$ is trivial.
\end{corollary}

\begin{proof} The assumption that $p \colon \sL \to \sA$  commutes  with  adjoint Lie actions of $\sA$ means that $p([x,y])=[x, p(y)]$ for every $x \in \sA$ and $y \in \sL$.

Let $\nabla \colon \sA \to \END_{\K}(\sE)$ be a flat $\sA$-connection on $\sE$. The existence of an $\Oh_X$-linear projection  $p \colon \sL \to \sA$  commuting with anchor maps ensures that the composition
	$ \widetilde{\nabla}:= \nabla p \colon \sL \to \END_{\K}(\sE)$ is a connection. In fact, for $l \in \sL$, $f \in \Oh_X$ and $e \in \sE$,
	\[ \widetilde{\nabla}_l (fe) = \nabla_{p(l)} (fe)= a_{\sA} (p(l))(f)e + f \nabla_{p(l)} (e)= a_{\sL}(l)(f)e + f \widetilde{\nabla}_l (e).\]
For every $a\in \sA$ and every $l\in\sL$ we have 
\[ [\widetilde{\nabla}_a,\widetilde{\nabla}_l]=[\nabla_{a},\nabla_{p(l)}]=\nabla_{[a,p(l)]}=\nabla_{p[a,l]}=
 \widetilde{\nabla}_{[a,l]},\]
 and this implies that the curvature of $\widetilde{\nabla}$ belongs to $\sG_2^2 \otimes \END_{\Oh_X}(\sE)$, so that by Theorem~\ref{thm.ostruzione_connessione_piatta} the Atiyah class of $\sE$ is trivial.
\end{proof}

Notice that Corollary~\ref{cor.criteriosvanimento} applies in particular in the case $X=\Spec(\K)$ and $\sA$ a semisimple Lie algebra. On the other hand, 
the Examples 2.10 and 2.11 of \cite{CSX} give explicit situations where $X$ is a single point and the 
Atiyah class does not vanish.

\bigskip
\section{Semiregularity maps and obstructions}

Let  $(\sL,\sA)$ be a Lie pair on a smooth separated scheme $X$ of finite type over a field $\K$ of characteristic 0.
Given a locally free $\sA$-module $(\sE, \nabla^{\sA})$ we introduced the Atiyah class 
\[  \At_{\sL/\sA}(\sE)\in \HH^1(\sA;(\sL/\sA)^\vee\otimes \END_{\Oh_X}(\sE)),\]
which is the \emph{primary obstruction} to the extension of the $\sA$-connection $\nabla^{\sA}$ to a flat $\sL$-connection; more precisely  the Atiyah class is a complete obstruction to the extension of  $\nabla^{\sA}$ to an  $\sL$-connection
with curvature in $\sG_2^2\otimes \END_{\Oh_X}(\sE)$.

Taking exterior cup products in $\sA$-cohomology it makes sense to consider the exterior powers
\[  \At_{\sL/\sA}(\sE)^{k}\in \HH^{k}\left(\sA;{\bigwedge}^k(\sL/\sA)^\vee\otimes 
\END_{\Oh_X}(\sE)\right)\]
together with the morphisms of graded vector spaces 
\[ \HH^{*}\left(\sA;\END_{\Oh_X}(\sE)\right)\to 
\HH^{*}\left(\sA;{\bigwedge}^k(\sL/\sA)^\vee\otimes \END_{\Oh_X}(\sE)\right)[k]\to 
\HH^{*}\left(\sA;{\bigwedge}^k(\sL/\sA)^\vee\right)[k],\]
\[ x\mapsto \frac{1}{k!}\Tr(\At_{\sL/\sA}(\sE)^{k}x).\]

The following definition is a clear natural extension of the definition of semiregularity maps for coherent sheaves \cite{ChS,BF}.

\begin{definition}  In the above situation, for every $k\ge 0$ the map 
\[ \tau_k\colon  \HH^{2}\left(\sA;\END_{\Oh_X}(\sE)\right)\to  
\HH^{2+k}\left(\sA;{\bigwedge}^k(\sL/\sA)^\vee\right),\quad
\tau_k(x)=\frac{1}{k!}\Tr(\At_{\sL/\sA}(\sE)^kx),\]
is called the \emph{$k$-semiregularity map of the $\sA$-module $(\sE, \nabla^{\sA})$}, (with respect to the Lie pair $(\sL,
\sA)$).
\end{definition}

If $\sG_*^*$ is the Leray filtration of the Lie pair $(\sL,
\sA)$ we have proved in Lemma~\ref{lem.bott0} that there exist canonical isomorphisms $\HH^{2+k}\left(\sA;{\bigwedge}^k(\sL/\sA)^\vee\right)\cong
\HH^{2+2k}\left(X,\sG_k^*/\sG_{k+1}^*\right)$ and therefore there exist  natural maps 
\[ i_k\colon \HH^{2+k}\left(\sA;{\bigwedge}^k(\sL/\sA)^\vee\right)\to 
\HH^{2+2k}\left(X,\frac{\Omega^*(\sL)}{\sG_{k+1}^*}\right),\]
which are injective whenever the Leray spectral sequence degenerates at $E_1$.

We are now ready to apply the abstract general results of \cite{ChS} to our situation in order to obtain the following result.

\begin{theorem} 
Let  $(\sL,\sA)$ be a Lie pair on a smooth separated scheme $X$ of finite type over a field $\K$ of characteristic 0.
Given a locally free $\sA$-module $(\sE, \nabla^{\sA})$, for every $k\ge 0$ the composite map 
\[ i_k\tau_k\colon  \HH^{2}\left(\sA;\END_{\Oh_X}(\sE)\right)\to \HH^{2+2k}\left(X,\frac{\Omega^*(\sL)}{\sG_{k+1}^*}\right)\]
annihilates every obstruction to deformations of $(\sE, \nabla^{\sA})$ as an $\sA$-module.
In particular, if the Leray spectral sequence of the Lie pair $(\sL,\sA)$ degenerates at $E_1$, then every semiregularity map 
annihilates obstructions.
\end{theorem}

\begin{proof} We take an affine cover $\sU$ of $X$ and we choose a simplicial connection
$\nabla \in \Tot(\sU, \Omega^1(\sL)\otimes P(\sL,\sE))$  extending $\nabla^{\sA}$.
By Lemma~\ref{lem.idealecurvato}, the ideal $I:= \Tot(\sU, \sG^*_1 \otimes \END_{\Oh_X}(\sE))$ is a curved ideal of the curved DG-algebra
\[A:=(\Tot(\sU,\Omega^*(\sL) \otimes \END_{\Oh_X}(\sE)), d_{\Tot} + d_{\nabla}, d_{\Tot}\nabla+ C), \]
so that the quotient 
\[ B:=A/I = \Tot(\sU, \Omega^*(\sA) \otimes \END_{\Oh_X}(\sE)) \]
is a non-curved DG-Lie algebra, with differential given by $d_{\Tot} + d_{\nabla^{\sA}}$. This is precisely the DG-Lie algebra controlling deformations of the $\sA$-module  $(\sE, \nabla^{\sA})$  of Theorem~\ref{thm.Deformazioni}.

The trace morphism
\[\Tr\colon \Omega^*(\sL,\END_{\sO_X}(\sE))\to\Omega^*(\sL)\]
of \eqref{eq.trace} induces
\[\Tr\colon \Tot(\sU, \Omega^*(\sL,\END_{\sO_X}(\sE)))\to \Tot(\sU, \Omega^*(\sL)),\]
which is a trace map in the sense of Definition \ref{def.tracemap}.  It is plain that 
\[\Tr (\Tot (\sU, \sG^*_k \otimes \END_{\sO_X}(\sE)))\subset \Tot(\sU, \sG^*_k),\]
for every $k\ge 0$. Finally, according to \eqref{equ.potenzeideale} and the exactness properties of $\Tot$, for every $i\le j$ we have 
\[ \frac{I^{(i)}}{I^{(j)}}=\frac{\Tot (\sU, \sG^*_i \otimes \END_{\sO_X}(\sE))}{\Tot (\sU, \sG^*_j \otimes \END_{\sO_X}(\sE))}=\Tot \left(\sU, \frac{\sG^*_i}{\sG_j^*} \otimes \END_{\sO_X}(\sE)\right).\]

Now, by  Theorem~\ref{thm.ChS}, there exists an $L_\infty$ morphism between DG-Lie algebras
\[ \sigma^k \colon \Tot(\sU, \Omega^*(\sA) \otimes \END_{\Oh_X}(\sE)) \rightsquigarrow \Tot\left(\sU, \frac{\Omega^*(\sL)}{\sG_{k+1}^*}[2k]\right)\]
whose linear component is given by
\[ \sigma^k_1 \colon \Tot(\sU, \Omega^*(\sA) \otimes \END_{\Oh_X}(\sE)) \to \Tot\left(\sU, \frac{\Omega^*(\sL)}{\sG_{k+1}^*}[2k]\right), \quad \sigma^k_1(x)= \frac{1}{k!}\Tr (R^k x),\]
where $R= d_{\Tot}\nabla+ C$ denotes the curvature of the DG-algebra $A$.

In cohomology the above maps $\sigma^k_1$ may be written as
	\[ \sigma^k_1
 \colon \HH^2(\sA;\END_{\Oh_X}(\sE)) \to \HH^{2k+2}\left(X,\frac{\Omega^*(\sL)}{\sG_{k+1}^*}\right) , \quad \sigma^k_1(x)= \frac{1}{k!}\Tr (\At_{\sL/\sA}(\sE)^k x),\] 
and then $\sigma^k_1=i_k\tau_k$.

Then the theorem is a consequence of the fact that the DG-Lie algebra $\Tot\left(\sU, \frac{\Omega^*(\sL)}{\sG_{k+1}^*}[2k]\right)$ is abelian and then, by general facts (see e.g. \cite{Man,LMDT}), every obstruction of the deformation functor associated to the DG-Lie algebra $B$ is annihilated by the maps $\sigma^k_1$.

\end{proof}

\begin{remark}
	The induced map in hypercohomology $\sigma^k_1$ depends only on the $\sA$-module $(\sE,\nabla^{\sA})$ and not on the choice of a simplicial  $\sL$-connection $\nabla$ extending $\nabla^{\sA}$. In fact, $\sigma^k_1$ depends only on the Atiyah class $\At_{\sL/\sA}(\sE)$ of the curved DG-pair 
	\[ (A=\Tot(\sU,\Omega^*(\sL, \END_{\Oh_X}(\sE))), I= \Tot(\sU,\sG^*_1\otimes 
	\END_{\sO_X}(\sE)) ),\]
	which we proved in Theorem~\ref{thm.ostruzione_connessione_piatta} does not depend on the choice of $\nabla$.
\end{remark}

\end{document}